\documentclass[12pt, reqno]{amsart}
\usepackage{amsmath, amsthm, amscd, amsfonts, amssymb, graphicx, color}
\usepackage[bookmarksnumbered, colorlinks, plainpages]{hyperref}

\hypersetup{colorlinks=true,linkcolor=red, anchorcolor=green, citecolor=cyan, urlcolor=red, filecolor=magenta, pdftoolbar=true}
\newtheorem{theorem}{Theorem}[section]
\newtheorem{lemma}[theorem]{Lemma}
\newtheorem{proposition}[theorem]{Proposition}
\newtheorem{corollary}[theorem]{Corollary}
\theoremstyle{definition}
\newtheorem{definition}[theorem]{Definition}

\theoremstyle{remark}

\numberwithin{equation}{section}

\begin{document}
\title[positive weakly $(q,r)$-dominated multilinear operators]{Tensorial
representations of positive weakly $(q,r)$-dominated multilinear operators}
\author{ Abdelaziz Belaada, Adel Bounabab, Athmane Ferradi and Khalil Saadi}
\date{}
\dedicatory{Laboratory of Functional Analysis and Geometry of Spaces,
Faculty of Mathematics and Computer Science, University of Mohamed
Boudiaf-M'sila, Po Box 166, Ichebilia, 28000, M'sila, Algeria.\\
Second address of the third author: Ecole Normale Sup\'{e}rieure de
Bousaada, Algeria\\
abdelaziz.belaada@univ-msila.dz\\
adel.bounabab@univ-msila.dz\\
ferradi.athmane@ens-bousaada.dz\\
khalil.saadi@univ-msila.dz}

\begin{abstract}
We introduce and study the class of positive weakly $(q,r)$-dominated
multilinear operators between Banach lattices. This notion extends classical
domination and summability concepts to the positive multilinear setting and
generates a new positive multi-ideal. A Pietsch domination theorem and a
polynomial version are established. Finally, we provide a tensorial
representation that yields an isometric identification with the dual of an
appropriate completed tensor product..
\end{abstract}

\maketitle

\setcounter{page}{1}


\let\thefootnote\relax\footnote{
\par
{}\textit{Key words and phrases. }Banach lattice, Positive multilinear
ideal, Positive polynomial ideal, Positive weakly $(q,r)$-dominated
multilinear operators, Positive weakly $(q,r)$-dominated polynomial,
Tensional representation}

\section{Introduction and preliminaries}

The theory of absolutely summing operators, initiated by Pietsch in the
1960s, has played a central role in the development of operator ideals and
their applications to Banach space theory. Since then, several nonlinear
extensions have been studied, especially for multilinear operators,
homogeneous polynomials, and Lipschitz mappings, leading to a rich framework
that unifies summability, domination, and factorization properties. This
area of research has provided a unified approach to extending linear results
to nonlinear settings, with significant contributions from works of Pietsch 
\cite{PIETSCHoi,PieMult}, Cohen \cite{Coh}, Kwapie\'{n} \cite
{Kwa1,Kwa2}, and others. In recent years, increasing attention has
focused on the positive versions of these operators. In fact, positive
operator theory, which uses the lattice structure of Banach spaces, has
become a powerful tool for strengthening and extending classical results. In 
\cite{FBS}, the basic elements of positive linear and multilinear operator
ideals were established. Thereafter, in \cite{Adel}, positive polynomial
ideals were introduced as a natural extension of the linear and multilinear
cases. These ideals not only encompass the corresponding positive classes,
which fail to satisfy the conditions of classical ideals, but also provide a
unifying framework for their study. A positive multilinear ideal $\mathcal{M}%
^{+}$ (or polynomial ideal $\mathcal{P}^{+}$) is a class of multilinear
operators (or polynomials) between Banach lattices that is stable under
composition with positive linear operators. The theory of summability in the
positive setting not only produces sharper inequalities but also exposes
phenomena absent in the purely linear framework. The concept of absolutely $%
\left( p,q_{1},\ldots,q_{m},r\right) $-summing multilinear operators,
introduced by Achour \cite{Ach11}, provides a natural extension of the
classical absolutely $(p,q,r)$-summing operators of Pietsch \cite{PIETSCHoi}%
. When $\frac{1}{p}=\frac{1}{q}+\frac{1}{r},$ this class is referred to as
the $\left( q,r\right) $-dominated operators. In this paper, we introduce
the class of positive weakly $(q, r)$-dominated multilinear
operators, which combines lattice positivity with weak absolute summability.
This class is stable under composition with positive operators and fits
naturally into the framework of positive multi-ideals. We also establish the
Pietsch domination theorem in this setting, characterizing these operators
through vector measures on the positive balls of the dual space. This result
extends the classical Pietsch theorem to the positive setting, based on
Banach lattice spaces. Furthermore, we also examine the polynomial version,
showing that the structure extends to $m$-homogeneous polynomials, giving
rise to positive polynomial ideals. Finally, we establish a tensorial
representation for positive weakly $(q,r)$-dominated multilinear
operators. By introducing a suitable tensor norm, we obtain an isometric
identification of these operators with the dual of a completed tensor
product, thereby providing the natural tensorial framework for the theory.
The same approach applies to polynomials, where a tensor norm is constructed
using $\widehat{\otimes }_{s,\left\vert \pi \right\vert }^{m}E,$ the $m$%
-fold positive projective symmetric tensor product of $E$.

The paper is organized as follows. we recall standard notations used
throughout the paper. We present Banach lattice spaces and some of their key
properties. We provide the definition of the regular multilinear space\ $%
\mathcal{L}^{r}\left( E_{1},\ldots ,E_{m};F\right) $ and regular polynomials
space\ $\mathcal{P}^{r}\left( ^{m}E;F\right) $ which are needed for defining
positive weakly  $(q,r)$-summing operators. Section 2 introduces
the class of positive weakly $(q, r)$-dominated multilinear
operators and establishes their basic properties. We then naturally extend
this to define positive weakly $(q, r)$-dominated polynomials.
Both classes form positive ideals. Section 3 is devoted to the tensorial
representation, which leads to the desired isometric identification. In the
case where $\frac{1}{p}=\frac{1}{q}+\frac{1}{r},$ we show that the space $%
\mathcal{L}_{w,\left( q;r\right) }^{m+}\left( E,\ldots ,E_{m};F\right) $ of
positive weakly $(q,r)$-dominated multilinear operators can be identified
with the dual of 
\begin{equation*}
E_{1}\widehat{\otimes }_{\mu _{\left( q;r\right) }^{m+}}\cdots \widehat{%
\otimes }_{\mu _{\left( q;r\right) }^{m+}}E_{m}\widehat{\otimes }_{\mu
_{\left( q;r\right) }^{m+}}F^{\ast },
\end{equation*}%
where $\mu _{\left( q;r\right) }^{m+}$ is a tensor norm that we define
below. Similarly, we identify the space $\mathcal{P}_{w,\left( q;r\right)
}^{m+}\left( ^{m}E;F\right) $ of positive weakly $(q,r)$-dominated
polynomials with 
\begin{equation*}
(\left( \widehat{\otimes }_{s,\left\vert \pi \right\vert }^{m}E\right) 
\widehat{\otimes }_{\lambda _{\left( q;r\right) }^{m+}}F^{\ast })^{\ast },
\end{equation*}%
where we show that $\lambda _{\left( q;r\right) }^{m+}$ is $\mu _{\left(
q;r\right) }^{1+}$.

Throughout the paper, $E,F$ and $G$ denote Banach lattices and $X,Y$ denote
Banach spaces. Our spaces are over the field of real scalars $\mathbb{R}$.
By $B_{X}$ we denote the closed unit ball of $X$ and by $X^{\ast }$ its
topological dual. We use the symbol $\mathcal{L}(X;Y)$ for the space of all
bounded linear operators from $X$ into $Y$.\ For $1\leq p\leq \infty $, we
denote by $p^{\ast }$ its conjugate, i.e., $1/p+1/p^{\ast }=1$. Let $E$ be a
Banach lattice with norm $\left\Vert \cdot \right\Vert $ and order $\leq $.
We denote by $E^{+}$ the positive cone of $E$, i.e., $E^{+}=\{x\in
E:x\geqslant 0\}.$ Let $x\in E,$ its positive part is defined by $%
x^{+}:=\sup \{x,0\}\geq 0\ $and its negative part is defined by $x^{-}:=\sup
\{-x,0\}\geq 0.$ We have $x=x^{+}-x^{-},$ $\left\vert x\right\vert
=x^{+}+x^{-},$ and the inequalities $x\leq \left\vert x\right\vert
,x^{+}\leq \left\vert x\right\vert $ and $x^{-}\leq \left\vert x\right\vert .
$ The dual $E^{\ast }$ of a Banach lattice $E$ is a Banach lattice with the
natural order $x_{1}^{\ast }\leq x_{2}^{\ast }\Leftrightarrow \langle
x,x_{1}^{\ast }\rangle \leq \langle x,x_{2}^{\ast }\rangle ,\forall x\in
E^{+}.$ Since $E$ is a sublattice of $E^{\ast \ast },$ we have for $%
x_{1},x_{2}\in E$ $x_{1}\leq x_{2}\Longleftrightarrow \left\langle
x_{1},x^{\ast }\right\rangle \leq \left\langle x_{2},x^{\ast }\right\rangle
,\quad \forall x^{\ast }\in E^{\ast +}.$ We have $\left\vert \left\langle
x^{\ast },x\right\rangle \right\vert \leq \left\langle \left\vert x^{\ast
}\right\vert ,\left\vert x\right\vert \right\rangle ,$ for every $x^{\ast
}\in E^{\ast }$ and $x\in E.$ We denote by $\ell _{p}^{n}(X)$ the Banach
space of all absolutely $p$-summable sequences $(x_{i})_{i=1}^{n}\subset X$
with the norm $\Vert (x_{i})_{i=1}^{n}\Vert _{p}=(\sum_{i=1}^{n}\Vert
x_{i}\Vert ^{p})^{\frac{1}{p}},$ and by $\ell _{p,w}^{n}(X)$ the Banach
space of all weakly $p$-summable sequences $(x_{i})_{i=1}^{n}\subset X$ with
the norm, $\Vert (x_{i})_{i=1}^{n}\Vert _{p,w}=\sup_{x^{\ast }\in B_{X^{\ast
}}}(\sum_{i=1}^{n}|\langle x^{\ast },x_{i}\rangle |^{p})^{\frac{1}{p}}.$
Consider the case where $X$ is replaced by a Banach lattice $E$, and define 
\begin{equation*}
\ell _{p,|w|}^{n}(E)=\{(x_{i})_{i=1}^{n}\subset E:\left( |x_{i}|\right)
_{i=1}^{n}\in \ell _{p,w}^{n}(E)\}\text{ and }\Vert (x_{i})_{i=1}^{n}\Vert
_{p,|w|}=\Vert (|x_{i}|)_{i=1}^{n}\Vert _{p,w}.
\end{equation*}%
Let $B_{E^{\ast }}^{+}=\left\{ x^{\ast }\in B_{E^{\ast }}:x^{\ast }\geq
0\right\} =B_{E^{\ast }}\cap E^{\ast +}$. If $(x_{i})_{i=1}^{n}\subset E^{+}$
, we have that 
\begin{equation*}
\Vert (x_{i})_{i=1}^{n}\Vert _{p,|w|}=\Vert (x_{i})_{i=1}^{n}\Vert
_{p,w}=\sup_{x^{\ast }\in B_{E^{\ast }}^{+}}(\sum_{i=1}^{n}\langle x^{\ast
},x_{i}\rangle ^{p})^{\frac{1}{p}}.
\end{equation*}%
For every $\left( x_{i}\right) _{i=1}^{n}\subset E,$ it is straightforward
to show that 
\begin{equation}
\Vert (x_{i}^{+})_{i=1}^{n}\Vert _{p,|w|}\leq \Vert (x_{i})_{i=1}^{n}\Vert
_{p,|w|}\text{ and }\Vert (x_{i}^{-})_{i=1}^{n}\Vert _{p,|w|}\leq \Vert
(x_{i})_{i=1}^{n}\Vert _{p,|w|}.  \label{122}
\end{equation}%
Given $m\in \mathbb{N}^{\ast },$ we denote by $\mathcal{L}(E_{1},\ldots,E_{m};F)
$ the Banach space of all bounded multilinear operators from $E_{1}\times
\cdots \times E_{n}$ into $F$ endowed with the supremum norm $\left\Vert
T\right\Vert =\sup_{\substack{ \left\Vert x_{j}\right\Vert \leq 1 \\ \left(
1\leq j\leq m\right) }}\left\Vert T\left( x_{1},\ldots,x_{m}\right) \right\Vert
.$ An operator $T\in \mathcal{L}(E_{1},\ldots,E_{m};F)$ is called positive if $%
T\left( x_{1},\ldots,x_{m}\right) \geq 0$ for every $x_{j}\in E_{j}^{+}\left(
1\leq j\leq m\right) .$ We denote by $\mathcal{L}^{+}(E_{1},\ldots,E_{m};F)$
the set of all positive $m$-linear operators. For every $T\in \mathcal{L}%
^{+}(E_{1},\ldots,E_{m};F)$ and $x_{j}\in E_{j}\left( 1\leq j\leq m\right) ,$
we have%
\begin{equation*}
\left\vert T\left( x_{1},\ldots,x_{m}\right) \right\vert \leq T\left(
\left\vert x_{1}\right\vert ,\ldots,\left\vert x_{m}\right\vert \right) .
\end{equation*}%
An $m$-linear operator $T:E_{1}\times \cdots \times E_{m}\rightarrow F$ is a
lattice $m$-morphism if 
\begin{equation*}
\left\vert T\left( x_{1},\ldots,x_{m}\right) \right\vert =T\left( \left\vert
x_{1}\right\vert ,\ldots,\left\vert x_{m}\right\vert \right) 
\end{equation*}%
for all $x_{j}\in E_{j}$. An $m$-linear operator $T:E_{1}\times \cdots
\times E_{m}\rightarrow F$, is called regular if it can be written as $%
T=T_{1}-T_{2}$ with $T_{1},T_{2}\in \mathcal{L}^{+}(E_{1},\ldots,E_{m};F).$ We
denote by $\mathcal{L}^{r}(E_{1},\ldots,E_{m};F)$ the space of all regular $m$%
-linear operators from $E_{1}\times \cdots \times E_{n}$ into $F.$ In \cite%
{Bu}, if $F$ is Dedekind complete, then $\mathcal{L}^{r}(E_{1},\ldots,E_{m};F)$
is a Banach lattice with the norm $\left\Vert T\right\Vert _{\mathcal{L}%
^{r}}=\left\Vert \left\vert T\right\vert \right\Vert .$ In this case, $%
\mathcal{L}^{r+}(E_{1},\ldots,E_{m};F)=\mathcal{L}^{+}(E_{1},\ldots,E_{m};F).$ For
every $x_{j}^{\ast }\in E_{j}^{\ast }\left( 1\leq j\leq m\right) $, we have $%
x_{1}^{\ast }\otimes \cdots \otimes x_{m}^{\ast }\in \mathcal{L}^{r}\left(
E_{1},\ldots,E_{m}\right) ,$ and 
\begin{equation*}
\left\Vert x_{1}^{\ast }\otimes \cdots \otimes x_{m}^{\ast }\right\Vert _{%
\mathcal{L}^{r}}=\left\Vert x_{1}^{\ast }\right\Vert \cdots \left\Vert
x_{m}^{\ast }\right\Vert .
\end{equation*}%
Let $E_{1},\ldots,E_{m}$ be Banach lattices, and let $E_{1}\otimes \cdots
\otimes E_{m}$ denote their algebraic tensor product$.$ Fremlin \cite{Fre}
introduced the vector lattice tensor product $E_{1}\overline{\otimes }\cdots 
\overline{\otimes }E_{m},$ defined so that 
\begin{equation*}
\left\vert x_{1}\otimes \cdots \otimes x_{m}\right\vert =\left\vert
x_{1}\right\vert \otimes \cdots \otimes \left\vert x_{m}\right\vert 
\end{equation*}%
for all $x_{j}\in E_{j}\left( 1\leq j\leq m\right) .$ He also introduced the
positive projective tensor product $E_{1}\otimes _{\left\vert \pi
\right\vert }\cdots \otimes _{\left\vert \pi \right\vert }E_{m},$ where for
every $\theta \in E_{1}\overline{\otimes }\cdots \overline{\otimes }E_{m}$ 
\begin{equation*}
\left\Vert \theta \right\Vert _{\left\vert \pi \right\vert }=\left\{
\sum_{i=1}^{k}\prod\limits_{j=1}^{m}\left\Vert x_{i}^{j}\right\Vert
:x_{i}^{j}\in E_{j}^{+},k\in \mathbb{N}^{\ast },\left\vert \theta
\right\vert \leq \sum_{i=1}^{k}x_{i}^{1}\otimes \cdots \otimes
x_{i}^{m}\right\} ,
\end{equation*}%
Its completion $E_{1}\widehat{\otimes }_{\left\vert \pi \right\vert }\cdots 
\widehat{\otimes }_{\left\vert \pi \right\vert }E_{m}$ is again a Banach
lattice, and the canonical mapping $\otimes \left( x_{1},\ldots,x_{m}\right)
\longmapsto x_{1}\otimes \cdots \otimes x_{m}$ is a lattice $m$-morphism. If 
$F$ is Dedekind complete, according to \cite[Proposition 3.3]{Bu}, every
regular $m$-linear operator $T:E_{1}\times \cdots \times E_{m}\rightarrow F$
admits a unique linearization $T^{\otimes }\in \mathcal{L}^{r}(E_{1}\widehat{%
\otimes }_{\left\vert \pi \right\vert }\cdots \widehat{\otimes }_{\left\vert
\pi \right\vert }E_{m};F)$ such that $T^{\otimes }\left( x_{1}\otimes \cdots
\otimes x_{m}\right) =T\left( x_{1},\ldots,x_{m}\right) ,$ yielding an
isometric lattice isomorphism between $\mathcal{L}^{r}(E_{1}\widehat{\otimes 
}_{\left\vert \pi \right\vert }\cdots \widehat{\otimes }_{\left\vert \pi
\right\vert }E_{m};F)$ and $\mathcal{L}^{r}(E_{1},\ldots,E_{m};F)$. In
particular, when $F=\mathbb{R},$ we have the isometrically isomorphic and
lattice homomorphic identification 
\begin{equation*}
\mathcal{L}^{r}(E_{1},\ldots,E_{m})=\left( E_{1}\widehat{\otimes }_{\left\vert
\pi \right\vert }\cdots \widehat{\otimes }_{\left\vert \pi \right\vert
}E_{m}\right) ^{\ast }.
\end{equation*}%
Consequently,%
\begin{equation*}
B_{\left( E_{1}\widehat{\otimes }_{\left\vert \pi \right\vert }\cdots\widehat{%
\otimes }_{\left\vert \pi \right\vert }E_{m}\right) ^{\ast }}^{+}=B_{%
\mathcal{L}^{r}(E_{1},\ldots,E_{m})}^{+}=\left\{ T\in \mathcal{L}%
^{+}(E_{1},\ldots,E_{m}):\left\Vert T\right\Vert \leq 1\right\} .
\end{equation*}%
Moreover, for every $\varphi \in \mathcal{L}^{r}(E_{1},\ldots,E_{m})$ and $%
x_{1}\otimes \cdots\otimes x_{m}\in E_{1}\otimes \cdots\otimes E_{m},$ we have%
\begin{equation*}
\left\vert \left\langle \varphi ,x_{1}\otimes \cdots \otimes
x_{m}\right\rangle \right\vert \leq \left\langle \left\vert \varphi
\right\vert ,\left\vert x_{1}\right\vert \otimes \cdots \otimes \left\vert
x_{m}\right\vert \right\rangle ,
\end{equation*}%
For $\epsilon _{1},\ldots,\epsilon _{m}\in \left\{ +,-\right\} ,$ we have%
\begin{equation}
\sup_{\varphi \in B_{\mathcal{L}^{r}\left( E_{1},\ldots,E_{2}\right)
}^{+}}\left( \sum\limits_{i=1}^{n}\varphi (x_{i}^{1\epsilon
_{1}},\ldots,x_{i}^{n\epsilon _{m}})^{q}\right) ^{\frac{1}{q}}\leq
\sup_{\varphi \in B_{\mathcal{L}^{r}\left( E_{1},\ldots,E_{2}\right)
}^{+}}\left( \sum\limits_{i=1}^{n}\varphi (\left\vert x_{i}^{1}\right\vert
,\ldots,\left\vert x_{i}^{n}\right\vert )^{q}\right) ^{\frac{1}{q}}.
\label{124}
\end{equation}%
A map $P:X\rightarrow Y$ is an $m$-homogeneous polynomial if there exists a
unique symmetric $m$-linear operator $\widehat{P}:X\times \overset{(m)}{\cdots}%
\times X\rightarrow Y$ such that $P\left( x\right) =\widehat{P}\left( x,%
\overset{(m)}{\ldots},x\right) .$ We denote by $\mathcal{P}\left(
^{m}X;Y\right) $, the Banach space of all continuous $m$-homogeneous
polynomials from $X$ into $Y$ endowed with the norm%
\begin{equation*}
\left\Vert P\right\Vert =\sup_{\left\Vert x\right\Vert \leq 1}\left\Vert
P\left( x\right) \right\Vert =\inf \left\{ C:\left\Vert P\left( x\right)
\right\Vert \leq C\left\Vert x\right\Vert ^{m},x\in X\right\} .
\end{equation*}%
We denote by $\mathcal{P}_{f}(^{m}X;Y)$ the space of all $m$-homogeneous
polynomials of finite type, that is%
\begin{equation*}
\mathcal{P}_{f}(^{m}X;Y)=\left\{ \sum\limits_{i=1}^{k}\varphi
_{i}^{m}\left( x\right) y_{i}:k\in \mathbb{N},\varphi _{i}\in X^{\ast
},y_{i}\in Y,1\leq i\leq k\right\} .
\end{equation*}%
Let $E$ and $F$ be Banach lattices. An $m$-homogeneous polynomials $\mathcal{%
P}\left( ^{m}E;F\right) $ is called regular if its associated symmetric $m$%
-linear operator $\widehat{P}$ is regular. We denote by $\mathcal{P}%
^{r}(^{m}E;F)$ the space of all regular polynomials from $E$ into $F.$ It is
easy to see that $P$ is regular if and only if there exist $P_{1},P_{2}\in 
\mathcal{P}^{+}(^{m}E;F)$ such that $P=P_{1}-P_{2}.$ For a Banach lattice $E$%
, the positive projective symmetric tensor norm on $\overline{\otimes }%
_{s}^{m}E$ is defined by%
\begin{equation*}
\left\Vert u\right\Vert _{s,\left\vert \pi \right\vert }=\inf \left\{
\sum\limits_{i=1}^{k}\left\Vert x_{i}\right\Vert ^{m}:x_{i}\in E^{+},k\in 
\mathbb{N}^{\ast },\left\vert u\right\vert \leq
\sum\limits_{i=1}^{k}x_{i}\otimes \overset{\left( m\right) }{\cdots}\otimes
x_{i}\right\} 
\end{equation*}%
for each $u\in \overline{\otimes }_{s}^{m}E.$ We denote by $\widehat{\otimes 
}_{s,\left\vert \pi \right\vert }^{m}E$ the completion of $\overline{\otimes 
}_{s}^{m}E$ under the lattice norm $\left\Vert \cdot \right\Vert
_{s,\left\vert \pi \right\vert }.$ Then $\widehat{\otimes }_{s,\left\vert
\pi \right\vert }^{m}E$ is a Banach lattice, called the $m$-fold positive
projective symmetric tensor product of $E.$ Moreover, if $F$ is Dedekind
complete Banach lattice then for any regular $m$-homogeneous polynomial $%
P:E\rightarrow F$ there exists a unique regular linear operator $P^{\otimes
}:\widehat{\otimes }_{s,\left\vert \pi \right\vert }^{m}E\rightarrow F,$
called the linearization of $P,$ such that $P\left( x\right) =P^{\otimes
}\left( x\otimes \overset{\left( m\right) }{\cdots}\otimes x\right) $ for every 
$x\in E.$ Moreover, in \cite[Proposition 3.4]{Bu}, the correspondence $%
P\mapsto P^{\otimes }$ is isometrically isomorphic and lattice homomorphic
between the Banach lattices $\mathcal{P}^{r}(^{m}E;F)$ and $\mathcal{L}%
^{r}\left( \widehat{\otimes }_{s,\left\vert \pi \right\vert }^{m}E;F\right) $%
. If $F=\mathbb{R},$ we have 
\begin{equation*}
\mathcal{P}^{r}\left( ^{m}E\right) =\left( \widehat{\otimes }_{s,\left\vert
\pi \right\vert }^{m}E\right) ^{\ast }.
\end{equation*}

In \cite{FBS}, the authors give the following definition: A positive
multi-ideal is a subclass $\mathcal{M}^{+}$ of all continuous multilinear
operators between Banach lattices such that for all $m\in \mathbb{N}^{\ast }$
and Banach lattices $E_{1},\ldots ,E_{m}$ and $F$, the components 
\begin{equation*}
\mathcal{M}^{+}(E_{1},\ldots,E_{m};F):=\mathcal{L}(E_{1},\ldots,E_{m};F)\cap 
\mathcal{M}^{+}
\end{equation*}%
satisfy:\newline
$(i)$ $\mathcal{M}^{+}(E_{1},\ldots,E_{m};F)$ is a linear subspace of $\mathcal{%
L}(E_{1},\ldots,E_{m};F)$ which contains the $m$-linear mappings of finite rank.%
\newline
$(ii)$ The positive ideal property: If $T\in \mathcal{M}^{+}\left(
E_{1},\ldots ,E_{m};F\right) ,u_{j}\in \mathcal{L}^{+}\left(
G_{j};E_{j}\right) $ for $j=$ $1,\ldots ,m$ and $v\in \mathcal{L}^{+}(F;G)$,
then $v\circ T\circ \left( u_{1},\ldots ,u_{m}\right) $ is in $\mathcal{M}%
^{+}\left( G_{1},\ldots ,G_{m};G\right) $.\newline
If $\Vert \cdot \Vert _{\mathcal{M}^{+}}:\mathcal{M}^{+}\rightarrow \mathbb{R%
}^{+}$ satisfies:\newline
a) $\left( \mathcal{M}^{+}(E_{1},\ldots,E_{m};F),\Vert \cdot \Vert _{\mathcal{M}%
^{+}}\right) $ is a Banach space for all Banach lattices $E_{1},\ldots
,E_{m},$ $F$.\newline
b) The canonical $m$-linear form $T^{m}:\mathbb{R}^{m}\rightarrow \mathbb{R}$
given by $T^{m}\left( \lambda ^{1},\ldots ,\lambda ^{m}\right) =\lambda
^{1}\cdots \lambda ^{m}$ satisfies $\left\Vert T^{m}\right\Vert _{\mathcal{M}%
^{+}}=1$ for all $m$,\newline
c) $T\in \mathcal{M}^{+}\left( E_{1},\ldots ,E_{m};F\right) ,u_{j}\in 
\mathcal{L}^{+}\left( G_{j};E_{j}\right) $ for $j=1,\ldots ,m$ and $v\in 
\mathcal{L}^{+}(F;G)$ then 
\begin{equation*}
\left\Vert v\circ T\circ \left( u_{1},\ldots ,u_{m}\right) \right\Vert _{%
\mathcal{M}^{+}}\leq \Vert v\Vert \Vert T\Vert _{\mathcal{M}^{+}}\left\Vert
u_{1}\right\Vert \cdots \left\Vert u_{m}\right\Vert .
\end{equation*}%
The class $\left( \mathcal{M}^{+},\Vert \cdot \Vert _{\mathcal{M}}\right) $
is referred to as a positive Banach multi-ideal. In particular, when $m=1$,
we specifically refer to it as a positive Banach ideal. Replacing the class $%
\mathcal{M}^{+}$ with the polynomial class $\mathcal{P}^{+}$, and condition
c) with: $P\in \mathcal{P}^{+}\left( ^{m}E;F\right) ,u\in \mathcal{L}%
^{+}\left( G;E\right) $ and $v\in \mathcal{L}^{+}(F;G)$ together with 
\begin{equation*}
\left\Vert v\circ P\circ u\right\Vert _{\mathcal{P}^{+}}\leq \Vert v\Vert
\Vert P\Vert _{\mathcal{P}^{+}}\left\Vert u\right\Vert ^{m},
\end{equation*}%
we obtain the definition of positive polynomial ideals.

\section{Positive weakly $(q,r)$-dominated multilinear operators}

The notion of absolutely $(p,q,r)$-summing linear operators was first
introduced and studied by Pietsch \cite{PIETSCHoi}. Later, Achour in \cite%
{Ach11} extended this concept to the multilinear setting by defining
absolutely $(p,q_{1},\ldots ,q_{m};r)$-summing multilinear operators. A
positive linear version was subsequently introduced and analyzed by Chen et
al. \cite{CBD21}. Building on these ideas, a positive multilinear version
was proposed in \cite{FBS}, adapting Achour's definition to the ordered
context. When $\frac{1}{p}=\frac{1}{q}+\frac{1}{r},$ this notion is referred
to as $(q,r)$-dominated linear operators. In this section, we present
another version of positive multilinear operators. Let us start with de
definition in the positive linear case.

\begin{definition}
\cite[Definition 3.1]{CBD21} Consider $1\leq r,p,q\leq \infty $ such that $%
\frac{1}{p}=\frac{1}{q}+\frac{1}{r}$. Let $E$ and $F$ be Banach lattices. A
mapping $u\in \mathcal{L}\left( E;F\right) $ is said to be positive $(q,r)$%
-dominated if there is a constant $C>0$ such that for every $%
x_{1},\ldots,x_{n}\in E$ and $y_{1}^{\ast },\ldots ,y_{n}^{\ast }\in F^{\ast }$%
, the following inequality holds: 
\begin{equation}
\left\Vert \left( \left\langle u\left( x_{i}\right) ,y_{i}^{\ast
}\right\rangle \right) _{i=1}^{n}\right\Vert _{p}\leq C\left\Vert \left(
x_{i}\right) _{i=1}^{n}\right\Vert _{q,\left\vert w\right\vert }\left\Vert
\left( y_{i}^{\ast }\right) _{i=1}^{n}\right\Vert _{r,\left\vert
w\right\vert }.  \label{2.1}
\end{equation}%
The space consisting of all such mappings is denoted by $\Psi _{\left(
q,r\right) }\left( E;F\right) $. In this case, we define 
\begin{equation*}
\left\Vert u\right\Vert _{\Psi _{\left( q,r\right) }}=\inf \{C>0:C\text{ 
\textit{satisfies }}(\ref{2.1})\}.
\end{equation*}
\end{definition}

There are several ways to generalize this definition in the positive
multilinear setting. In \cite{FBS}, the authors introduced a positive
multilinear version by adapting each variable to its corresponding weakly $%
q_{j}$-summable norm. In our approach, we adjust the variables to a single
regular form, which allows us to employ weakly $q$-summable sequences in the
space $\mathcal{L}^{r}\left( E_{1},\ldots ,E_{m}\right) $. This idea was
used by Belacel et al. \cite{BBH23} to define the concept of positive weakly
Cohen $p$-nuclear operators. We will show that this new class forms a
concrete example of a positive multi-ideal, satisfies Pietsch's domination
theorem, and admits a representation via a tensor product equipped with a
norm tailored to this class.

\begin{definition}
\label{def1}Consider $1\leq p,q,r\leq \infty $ such that $\frac{1}{p}=\frac{1%
}{q}+\frac{1}{r}$. Let $E_{1},\ldots,E_{m}$ and $F$ be Banach lattices. A
mapping $T\in \mathcal{L}\left( E_{1},\ldots ,E_{m};F\right) $ is said to be
positive weakly $(q,r)$-dominated if there is a constant $C>0$ such that for
every $(x_{i}^{1},\ldots,x_{i}^{m})\in E_{1}^{+}\times \cdots\times E_{m}^{+}$ $%
\left( 1\leq i\leq n\right) $ and $y_{1}^{\ast },\ldots ,y_{n}^{\ast }\in
F^{\ast +}$, the following inequality holds: 
\begin{equation}
\left\Vert \left( \left\langle T\left( x_{i}^{1},\ldots ,x_{i}^{m}\right)
,y_{i}^{\ast }\right\rangle \right) _{i=1}^{n}\right\Vert _{p}\leq
C\sup_{\varphi \in B_{\mathcal{L}^{r}\left( E_{1},\ldots ,E_{m}\right)
}^{+}}(\sum\limits_{i=1}^{n}\varphi \left( x_{i}^{1},\ldots
,x_{i}^{m}\right) ^{q})^{\frac{1}{q}}\left\Vert \left( y_{i}^{\ast }\right)
_{i=1}^{n}\right\Vert _{r,\left\vert w\right\vert }.  \label{2.2}
\end{equation}%
The space consisting of all such mappings is denoted by $\mathcal{L}%
_{w,\left( q,r\right) }^{m+}\left( E_{1},\ldots ,E_{m};F\right) $. In this
case, we define 
\begin{equation*}
d_{w,\left( q;r\right) }^{m+}(T)=\inf \{C>0:C\text{ \textit{satisfies }}(\ref%
{2.2})\}.
\end{equation*}
\end{definition}

It is straightforward to verify that every finite type multilinear operator
is positive weakly $(q, r)$-dominated. Hence,%
\begin{equation}
\mathcal{L}_{f}(E_{1},\ldots ,E_{m};F)\subset \mathcal{L}_{w,\left(
q,r\right) }^{m+}\left( E_{1},\ldots ,E_{m};F\right) .  \label{ranf}
\end{equation}%
For $m=1$, we obtain the following coincidence $\mathcal{L}_{w,\left(
q,r\right) }^{1+}\left( E_{1};F\right) =\Psi _{\left( q,r\right) }\left(
E_{1};F\right) $. In the next result, we give the following equivalent
definition.

\begin{proposition}
\label{Propo1}Let $1\leq p,q,r\leq \infty $ with $\frac{1}{p}=\frac{1}{q}+%
\frac{1}{r}$ and $T\in \mathcal{L}\left( E_{1},\ldots ,E_{m};F\right) $. The
following properties are equivalent:\newline
1) The operator $T$ is positive weakly $(q,r)$-dominated.\newline
2) There is a constant $C>0$ such that for any $(x_{i}^{1},\ldots,x_{i}^{m})\in
E_{1}\times \cdots\times E_{m}$ $\left( 1\leq i\leq n\right) $ and $y_{1}^{\ast
},\ldots ,y_{n}^{\ast }\in F^{\ast }$, we have 
\begin{eqnarray}
&&\left\Vert \left( \left\langle T\left( x_{i}^{1},\ldots ,x_{i}^{m}\right)
,y_{i}^{\ast }\right\rangle \right) _{i=1}^{n}\right\Vert _{p}
\label{def2sec3} \\
&\leq &C\sup_{\varphi \in B_{\mathcal{L}^{r}\left( E_{1},\ldots
,E_{m}\right) }^{+}}\left( \sum\limits_{i=1}^{n}\varphi \left( \left\vert
x_{i}^{1}\right\vert ,\ldots ,\left\vert x_{i}^{m}\right\vert \right)
^{q}\right) ^{\frac{1}{q}}\left\Vert \left( y_{i}^{\ast }\right)
_{i=1}^{n}\right\Vert _{r,\left\vert w\right\vert }.  \notag
\end{eqnarray}%
In this case, we define 
\begin{equation*}
d_{w,\left( q;r\right) }^{m+}(T)=\inf \{C>0:C\text{ \textit{satisfies }}(\ref%
{def2sec3})\}.
\end{equation*}
\end{proposition}

\begin{proof}
$2)\Rightarrow 1):$ Immediately applying Definition \ref{def1} for $%
(x_{i}^{1},\ldots,x_{i}^{m})\in E_{1}^{+}\times \cdots\times E_{m}^{+}$, $1\leq
i\leq n$ and $y_{1}^{\ast },\ldots,y_{n}^{\ast }\in F^{\ast +}$.\newline
$1)\Rightarrow 2):$ Suppose that $T$ is positive weakly $(q,r)$-dominated.
For convenience, we prove only the inequality for the case when $m=2$. Let $%
\left( x_{i}^{1},x_{i}^{2}\right) \in E_{1}\times E_{2},(1\leq i\leq n)$ $%
y_{1}^{\ast },\ldots ,y_{n}^{\ast }\in F^{\ast }$, then one has%
\begin{eqnarray*}
&&(\sum_{i=1}^{n}\left\vert \left\langle T(x_{i}^{1},x_{i}^{2}),y_{i}^{\ast
}\right\rangle \right\vert ^{p})^{\frac{1}{p}} \\
&=&(\sum_{i=1}^{n}\left\vert \left\langle T\left(
x_{i}^{1+}-x_{i}^{1-},x_{i}^{2+}-x_{i}^{2-}\right) ,y_{i}^{\ast
}\right\rangle \right\vert ^{p})^{\frac{1}{p}} \\
&\leq &(\sum_{i=1}^{n}\left\vert \left\langle T\left(
x_{i}^{1+},x_{i}^{2+}\right) ,y_{i}^{\ast }\right\rangle \right\vert ^{p})^{%
\frac{1}{p}}+(\sum_{i=1}^{n}\left\vert \left\langle T\left(
x_{i}^{1+},x_{i}^{2-}\right) ,y_{i}^{\ast }\right\rangle \right\vert ^{p})^{%
\frac{1}{p}}+ \\
&&(\sum_{i=1}^{n}\left\vert \left\langle T\left(
x_{i}^{1-},x_{i}^{2+}\right) ,y_{i}^{\ast }\right\rangle \right\vert ^{p})^{%
\frac{1}{p}}+(\sum_{i=1}^{n}\left\vert \left\langle T\left(
x_{i}^{1-},x_{i}^{2-}\right) ,y_{i}^{\ast }\right\rangle \right\vert ^{p})^{%
\frac{1}{p}},
\end{eqnarray*}%
which is less than or equal to%
\begin{eqnarray*}
&\leq &(\sum_{i=1}^{n}\left\vert \left\langle T\left(
x_{i}^{1+},x_{i}^{2+}\right) ,y_{i}^{\ast +}\right\rangle \right\vert ^{p})^{%
\frac{1}{p}}+(\sum_{i=1}^{n}\left\vert \left\langle T\left(
x_{i}^{1+},x_{i}^{2+}\right) ,y_{i}^{\ast -}\right\rangle \right\vert ^{p})^{%
\frac{1}{p}}+ \\
&&(\sum_{i=1}^{n}\left\vert \left\langle T\left(
x_{i}^{1+},x_{i}^{2-}\right) ,y_{i}^{\ast +}\right\rangle \right\vert ^{p})^{%
\frac{1}{p}}+(\sum_{i=1}^{n}\left\vert \left\langle T\left(
x_{i}^{1+},x_{i}^{2-}\right) ,y_{i}^{\ast -}\right\rangle \right\vert ^{p})^{%
\frac{1}{p}}+ \\
&&(\sum_{i=1}^{n}\left\vert \left\langle T\left(
x_{i}^{1-},x_{i}^{2+}\right) ,y_{i}^{\ast +}\right\rangle \right\vert ^{p})^{%
\frac{1}{p}}+(\sum_{i=1}^{n}\left\vert \left\langle T\left(
x_{i}^{1-},x_{i}^{2+}\right) ,y_{i}^{\ast -}\right\rangle \right\vert ^{p})^{%
\frac{1}{p}}+ \\
&&(\sum_{i=1}^{n}\left\vert \left\langle T\left(
x_{i}^{1-},x_{i}^{2+}\right) ,y_{i}^{\ast +}\right\rangle \right\vert ^{p})^{%
\frac{1}{p}}+(\sum_{i=1}^{n}\left\vert \left\langle T\left(
x_{i}^{1-},x_{i}^{2+}\right) ,y_{i}^{\ast -}\right\rangle \right\vert ^{p})^{%
\frac{1}{p}}+ \\
&&(\sum_{i=1}^{n}\left\vert \left\langle T\left(
x_{i}^{1-},x_{i}^{2-}\right) ,y_{i}^{\ast +}\right\rangle \right\vert ^{p})^{%
\frac{1}{p}}+(\sum_{i=1}^{n}\left\vert \left\langle T\left(
x_{i}^{1-},x_{i}^{2-}\right) ,y_{i}^{\ast -}\right\rangle \right\vert ^{p})^{%
\frac{1}{p}},
\end{eqnarray*}%
by using $(\ref{122})$ and $(\ref{124}),$ we obtain%
\begin{equation*}
(\sum_{i=1}^{n}\left\vert \left\langle T(x_{i}^{1},x_{i}^{2}),y_{i}^{\ast
}\right\rangle \right\vert ^{p})^{\frac{1}{p}}\leq 8d_{w,\left( q,r\right)
}^{2+}(T)\sup_{\varphi \in B_{\mathcal{L}^{r}\left( E_{1},E_{2}\right)
}^{+}}\left( \sum\limits_{i=1}^{n}\varphi (\left\vert x_{i}^{1}\right\vert
,\left\vert x_{i}^{2}\right\vert )^{q}\right) ^{\frac{1}{q}}\left\Vert
\left( y_{i}^{\ast }\right) _{i=1}^{n}\right\Vert _{r,\left\vert
w\right\vert }.
\end{equation*}
\end{proof}

\begin{proposition}
Let $E_{1},\ldots ,E_{m},F,G_{j}\left( 1\leq j\leq m\right) $ and $H$ are
Banach lattices. Let $T\in \mathcal{L}_{w,\left( q,r\right) }^{m+}\left(
E_{1},\ldots ,E_{m};F\right) ,$ $u_{j}\in \mathcal{L}^{+}\left(
G_{j};E_{j}\right) $ $\left( 1\leq j\leq m\right) $ and $v\in \mathcal{L}%
^{+}(F;H)$. Then%
\begin{equation*}
v\circ T\circ \left( u_{1},\ldots ,u_{m}\right) \in \mathcal{L}_{w,\left(
q;r\right) }^{m+}\left( G_{1},\ldots ,G_{m};H\right) .
\end{equation*}Moreover
\begin{equation*}
d_{w,\left( q,r\right) }^{m+}\left( v\circ T\circ \left( u_{1},\ldots
,u_{m}\right) \right) \leq d_{w,\left( q,r\right) }^{m+}(T)\left\Vert
u_{1}\right\Vert \cdots \left\Vert u_{m}\right\Vert \Vert v\Vert .
\end{equation*}

\end{proposition}

\begin{proof}
Let $(x_{i}^{1},\ldots,x_{i}^{m})\in G_{1}^{+}\times \cdots\times G_{m}^{+}$ $%
\left( 1\leq i\leq n\right) $ and $y_{1}^{\ast },\ldots ,y_{n}^{\ast }\in
G^{\ast +}.$ Since $T\in \mathcal{L}_{w,\left( q,r\right) }^{m+}\left(
E_{1},\ldots ,E_{m};F\right) ,$ $u_{j}\left( x_{i}^{j}\right) \geq 0$ and $%
v^{\ast }\left( y_{i}^{\ast }\right) \geq 0$ $\left( 1\leq j\leq m,1\leq
i\leq n\right) $ we have%
\begin{eqnarray*}
&&\left\Vert \left( \left\langle v\circ T\circ \left( u_{1},\ldots
,u_{m}\right) \left( x_{i}^{1},\ldots ,x_{i}^{m}\right) ,y_{i}^{\ast
}\right\rangle \right) _{i=1}^{n}\right\Vert _{p} \\
&=&\left\Vert \left( \left\langle T\left( u_{1}\left( x_{i}^{1}\right)
,\ldots ,u_{m}\left( x_{i}^{m}\right) \right) ,v^{\ast }\left( y_{i}^{\ast
}\right) \right\rangle \right) _{i=1}^{n}\right\Vert _{p} \\
&\leq &d_{w,\left( q,r\right) }^{m+}(T)\sup_{\varphi \in B_{\mathcal{L}%
^{r}\left( E_{1},\ldots ,E_{m}\right) }^{+}}(\sum\limits_{i=1}^{n}\varphi
\left( u_{1}\left( x_{i}^{1}\right) ,\ldots ,u_{m}\left( x_{i}^{m}\right)
\right) ^{q})^{\frac{1}{q}}\left\Vert \left( v^{\ast }\left( y_{i}^{\ast
}\right) \right) _{i=1}^{n}\right\Vert _{r,w} \\
&\leq &d_{w,\left( q,r\right) }^{m+}(T)\left\Vert u_{1}\right\Vert \cdots
\left\Vert u_{m}\right\Vert \Vert v\Vert \sup_{\varphi \in B_{\mathcal{L}%
^{r}\left( E_{1},\ldots ,E_{m}\right) }^{+}}(\sum\limits_{i=1}^{n}\varphi
\left( x_{i}^{1},\ldots ,x_{i}^{m}\right) ^{q})^{\frac{1}{q}}\left\Vert
\left( y_{i}^{\ast }\right) _{i=1}^{n}\right\Vert _{r,w}
\end{eqnarray*}%
thus $v\circ T\circ \left( u_{1},\ldots ,u_{m}\right) $ is in $\mathcal{L}%
_{w,\left( q,r\right) }^{m+}\left( G_{1},\ldots ,G_{m};G\right) $ and we have%
\begin{equation*}
d_{w,\left( q,r\right) }^{m+}\left( v\circ T\circ \left( u_{1},\ldots
,u_{m}\right) \right) \leq d_{w,\left( q,r\right) }^{m+}(T)\left\Vert
u_{1}\right\Vert \cdots \left\Vert u_{m}\right\Vert \Vert v\Vert .
\end{equation*}
\end{proof}

The pair $\left( \mathcal{L}_{w,\left( q,r\right) }^{m+},d_{d,\left(
q;r\right) }^{m+}\right) $ defines a positive Banach multilinear ideal. The
proof follows directly from the previous Proposition and the inclusion $(\ref%
{ranf})$, while the remaining details are straightforward. Now, we
characterize the positive weakly $(q,r)$-dominated multilinear operators by
the Pietsch domination theorem. For this purpose, we use the full general
Pietsch domination theorem given by Pellegrino et al. in \cite[Theorem 4.6]%
{PSS12} and \cite{ps}. For simplify, we denote by 
\begin{equation*}
\widehat{E}_{\left\vert \pi \right\vert }=E_{1}\widehat{\otimes }%
_{\left\vert \pi \right\vert }\cdots \widehat{\otimes }_{\left\vert \pi
\right\vert }E_{m}.
\end{equation*}

\begin{theorem}[Pietsch Domination Theorem]
\label{thdo1} Let $1\leq p,q,r\leq \infty $ with $\frac{1}{p}=\frac{1}{q}+%
\frac{1}{r}$. Let $E_{1},\ldots,E_{m}$ and $F$ be Banach lattices. The
following statements are equivalent.\newline
1) The operator $T\in \mathcal{L}\left( E_{1},\ldots ,E_{m};F\right) $ is
positive weakly $(q,r)$-dominated.\newline
2) There is a constant $C>0$ and Borel probability measures $\mu $ on $%
B_{\left( \widehat{E}_{\left\vert \pi \right\vert }\right) ^{\ast }}^{+}$
and $\eta $ on $B_{F^{\ast \ast }}^{+}$ such that%
\begin{equation}
\left\vert \langle T(x^{1},\ldots,x^{m}),y^{\ast }\rangle \right\vert \leq
C\left( \int_{B_{\left( \widehat{E}_{\left\vert \pi \right\vert }\right)
^{\ast }}^{+}}\varphi \left( \left\vert x_{i}^{1}\right\vert ,\ldots
,\left\vert x_{i}^{m}\right\vert \right) ^{q}d\mu \right) ^{\frac{1}{q}%
}(\int_{B_{F^{\ast \ast }}^{+}}\langle \left\vert y^{\ast }\right\vert
,y^{\ast \ast }\rangle ^{r}d\eta )^{\frac{1}{r}},  \label{123}
\end{equation}%
for all $(x^{1},\ldots,x^{m},y^{\ast })\in E_{1}\times \cdots\times E_{m}\times
F^{\ast }.$ Therefore, we have%
\begin{equation*}
d_{w,\left( q;r\right) }^{m+}(T)=\inf \{C>0:C\text{ \textit{satisfies
inequality} }(\ref{123})\}.
\end{equation*}%
3) There is a constant $C>0$ and Borel probability measures $\mu $ on $%
B_{\left( \widehat{E}_{\left\vert \pi \right\vert }\right) ^{\ast }}^{+}$
and $\eta $ on $B_{F^{\ast \ast }}^{+}$ such that 
\begin{equation}
\begin{array}{ll}
& |\langle T(x^{1},\ldots,x^{m}),y^{\ast }\rangle | \\ 
& \leq C\left( \int_{B_{\left( \widehat{E}_{\left\vert \pi \right\vert
}\right) ^{\ast }}^{+}}\varphi \left( x_{i}^{1},\ldots ,x_{i}^{m}\right)
^{q}d\mu \right) ^{\frac{1}{q}}\left( \int_{B_{F^{\ast \ast }}^{+}}\langle
y^{\ast },y^{\ast \ast }\rangle ^{r}d\eta \right) ^{\frac{1}{r}},%
\end{array}
\label{def2sec4}
\end{equation}%
for all $(x^{1},\ldots,x^{m},y^{\ast })\in E_{1}^{+}\times \cdots \times
E_{m}^{+}\times F^{\ast +}.$ Therefore, we have%
\begin{equation*}
d_{w,\left( q;r\right) }^{m+}(T)=\inf \{C>0:C\text{ \textit{satisfies }}(\ref%
{def2sec4})\}.
\end{equation*}
\end{theorem}

\begin{proof}
$1)\Longleftrightarrow 2):$ We will choose the parameters as specified in 
\cite[Theorem 4.6]{PSS12} 
\begin{equation*}
\left\{ 
\begin{array}{l}
S:\mathcal{L}\left( E_{1},\ldots ,E_{m};F\right) \times \left( E_{1}\times
\ldots \times E_{m}\times F^{\ast }\right) \times \mathbb{R\times R}%
\rightarrow \mathbb{R}^{+}: \\ 
S\left( T,\left( x^{1},\ldots,x^{m},y^{\ast }\right) ,\lambda _{1},\lambda
_{2}\right) =\left\vert \lambda _{2}\right\vert |\langle
T(x^{1},\ldots,x^{m}),y^{\ast }\rangle | \\ 
R_{1}:B_{\left( \widehat{E}_{\left\vert \pi \right\vert }\right) ^{\ast
}}^{+}\times \left( E_{1}\times \ldots \times E_{m}\times F^{\ast }\right)
\times \mathbb{R}\rightarrow \mathbb{R}^{+}: \\ 
R_{1}(\varphi ,\left( x^{1},\ldots,x^{m},y^{\ast }\right) ,\lambda
_{1})=\varphi \left( \left\vert x^{1}\right\vert ,\ldots ,\left\vert
x^{m}\right\vert \right)  \\ 
R_{2}:B_{F^{\ast \ast }}^{+}\times \left( E_{1}\times \cdots \times
E_{m}\times F^{\ast }\right) \times \mathbb{R}\rightarrow \mathbb{R}^{+}: \\ 
R_{2}(y^{\ast \ast },\left( x^{1},\ldots,x^{m},y^{\ast }\right) ,\lambda
_{2})=\left\vert \lambda _{2}\right\vert \langle |y^{\ast }|,y^{\ast \ast
}\rangle .%
\end{array}%
\right. 
\end{equation*}%
These maps satisfy conditions $\left( 1\right) $ and $\left( 2\right) $ from 
\cite[p. 1255]{PSS12}. We can easily conclude that $T:E_{1}\times \cdots
\times E_{m}\rightarrow F$ is positive weakly $(q,r)$-dominated if, and only
if, 
\begin{eqnarray*}
&&(\sum\limits_{i=1}^{n}S\left( T,\left(
x_{i}^{1},\ldots,x_{i}^{m},y_{i}^{\ast }\right) ,\lambda _{i,1},\lambda
_{i,2}\right) ^{p})^{\frac{1}{p}} \\
&\leq &C\sup_{\varphi \in B_{\left( \widehat{E}_{\left\vert \pi \right\vert
}\right) ^{\ast }}^{+}}(\sum\limits_{i=1}^{n}R_{1}(x^{\ast },\left(
x_{i}^{1},\ldots,x_{i}^{m},y_{i}^{\ast }\right) ,\lambda _{i,1})^{q})^{\frac{1}{%
q}} \\
&&\times \sup_{y^{\ast \ast }\in B_{F^{\ast \ast
}}^{+}}(\sum\limits_{i=1}^{n}R_{2}(y^{\ast \ast },\left(
x_{i}^{1},\ldots,x_{i}^{m},y_{i}^{\ast }\right) ,\lambda _{i,2})^{r})^{\frac{1}{%
r}},
\end{eqnarray*}%
i.e., $T$ is $R_{1},R_{2}$-$S$-abstract $(q,r)$-summing. As outlined in \cite%
[Theorem 4.6]{PSS12}, this implies that $T$ is $R_{1},R_{2}$-$S$-abstract $%
(q,r)$-summing if, and only if, there exists a positive constant $C$ and
probability measures $\mu $ on $B_{\left( \widehat{E}_{\left\vert \pi
\right\vert }\right) ^{\ast }}^{+}$ and $\eta $ on $B_{F^{\ast \ast }}^{+}$,
such that 
\begin{eqnarray*}
&&S\left( T,\left( x^{1},\ldots,x^{m},y^{\ast }\right) ,\lambda _{1},\lambda
_{2}\right)  \\
&\leq &C(\int_{B_{E^{\ast }}^{+}}R_{1}(x^{\ast },\left(
x^{1},\ldots,x^{m},y^{\ast }\right) ,\lambda _{1})^{q}d\mu )^{\frac{1}{q}%
}(\int_{B_{F^{\ast \ast }}^{+}}R_{2}(y^{\ast \ast },\left(
x^{1},\ldots,x^{m},y^{\ast }\right) ,\lambda _{2})^{r}d\eta )^{\frac{1}{r}}.
\end{eqnarray*}%
Consequently 
\begin{equation*}
\left\vert \langle T(x^{1},\ldots,x^{m}),y^{\ast }\rangle \right\vert \leq
C\left( \int_{B_{\left( \widehat{E}_{\left\vert \pi \right\vert }\right)
^{\ast }}^{+}}\varphi \left( \left\vert x^{1}\right\vert ,\ldots ,\left\vert
x^{m}\right\vert \right) ^{q}d\mu \right) ^{\frac{1}{q}}(\int_{B_{F^{\ast
\ast }}^{+}}\langle \left\vert y^{\ast }\right\vert ,y^{\ast \ast }\rangle
^{r}d\eta )^{\frac{1}{r}}.
\end{equation*}%
The implications $2)\Longrightarrow 3)$ and $3)\Longrightarrow 1$ are
immediate.
\end{proof}

\begin{proposition}
Let $T\in \mathcal{L}\left( E_{1},\ldots ,E_{m};F\right) $ and $1\leq
p,q,r\leq \infty $. Consider the following statements:\newline
1) $T^{\otimes }:E_{1}\widehat{\otimes }_{\left\vert \pi \right\vert }\cdots 
\widehat{\otimes }_{\left\vert \pi \right\vert }E_{m}\rightarrow F$ is
positive weakly $(q, r)$-dominated;\newline
2) There exist a Banach space $G$, a positive (Dimant) strongly $p$-summing $%
m$-linear operator $S:E_{1}\times \cdots \times E_{m}\rightarrow G$ and a
Cohen positive strongly $p$-summing linear operator $u:G\rightarrow F$ such
that $T=u\circ S$;\newline
3) $T$ is positive weakly $\left( q,r\right) $-dominated.\newline
Then, the statement 1) implies 2), which implies 3).
\end{proposition}

\begin{proof}
$1)\Rightarrow 2)$ Since $T^{\otimes }$ is positive $\left( q,r\right) $%
-dominated, by \cite[Theorem 3.7]{CBD21} there exist a Banach space $G$, an
positive $q$-summing linear operator $v:E_{1}\widehat{\otimes }_{\left\vert
\pi \right\vert }\cdots \widehat{\otimes }_{\left\vert \pi \right\vert
}E_{m}\rightarrow G$ and a Cohen positive strongly $q^{\ast }$-summing
linear operator $u:G\rightarrow F$ such that $T=u\circ v$. Let $S=v\circ
\otimes $.Then $T=u\circ v\circ \otimes =u\circ S$ and the following diagram 
\begin{equation*}
\begin{array}{ccc}
E_{1}\times \cdots \times E_{m} & \overset{T}{\longrightarrow } & F \\ 
\downarrow \otimes  & \searrow S & \uparrow u \\ 
\widehat{E}_{\left\vert \pi \right\vert } & \overset{v}{\longrightarrow } & G%
\end{array}%
\end{equation*}%
commutes. Since $S^{\otimes }=v$, by \cite[Corollary 1]{BBH2020}, it follows
that $S$ is positive Dimant strongly $p$-summing.

$2)\Longrightarrow 3):$ There exist a Borel probability measures $\mu $ on $%
B_{\left( \widehat{E}_{\left\vert \pi \right\vert }\right) ^{\ast }}^{+}$
and $\eta $ on $B_{F^{\ast \ast }}^{+}$ such that for every $%
(x^{1},\ldots,x^{m})\in E_{1}^{+}\times \cdots \times E_{m}^{+}$ and $y^{\ast
}\in F^{\ast +}$, we have 
\begin{eqnarray*}
&&\left\vert \langle T(x^{1},\ldots,x^{m}),y^{\ast }\rangle \right\vert  \\
&=&\left\langle u\circ S\left( x^{1},\ldots,x^{m}\right) ,y^{\ast
}\right\rangle  \\
&\leq &d_{p}^{+}\left( u\right) \left\Vert S\left( x^{1},\ldots,x^{m}\right)
\right\Vert \left( \int_{B_{F^{\ast \ast }}^{+}}\langle y^{\ast },y^{\ast
\ast }\rangle ^{r}d\eta \right) ^{\frac{1}{r}} \\
&\leq &d_{p}^{+}\left( u\right) d_{s,p}^{m+}\left( S\right) \left(
\int_{B_{\left( \widehat{E}_{\left\vert \pi \right\vert }\right) ^{\ast
}}^{+}}\varphi \left( x^{1},\ldots,x^{m}\right) ^{q}d\mu \right) ^{\frac{1}{q}%
}\left( \int_{B_{F^{\ast \ast }}^{+}}\langle y^{\ast },y^{\ast \ast }\rangle
^{r}d\eta \right) ^{\frac{1}{r}}.
\end{eqnarray*}
\end{proof}

In the sequel, we develop and analyze the corresponding positive polynomial
version. This approach allows us to establish that the class of positive
weakly $(q,r)$-dominated polynomials forms a positive polynomial ideal,
stable under composition with bounded positive operators. Moreover, we
establish a Pietsch domination theorem in the polynomial setting, showing
that the domination inequalities extend naturally from the multilinear case.

\begin{definition}
\label{definition7}Let $m\in \mathbb{N}^{\ast }.$ Let $1\leq p,q,r\leq
\infty $ with $\frac{1}{p}=\frac{1}{q}+\frac{1}{r}$. Let $E$ and $F$ be
Banach lattices. A polynomial $P\in \mathcal{P}\left( ^{m}E;F\right) $ is
called positive weakly $(q, r)$-dominated if there exists a constant $C>0$ such that
for any $(x_{i})_{i=1}^{n}\subset E^{+}$ and $\left( y_{i}^{\ast }\right)
_{i=1}^{n}\subset F^{\ast +}$, the following inequality holds: 
\begin{equation}
\left\Vert \left( \left\langle P\left( x_{i}\right) ,y_{i}^{\ast
}\right\rangle \right) _{i=1}^{n}\right\Vert _{p}\leq C\sup_{\phi \in B_{%
\mathcal{P}^{r}\left( ^{m}E\right) }^{+}}\left( \sum\limits_{i=1}^{n}\phi
\left( x_{i}\right) ^{q}\right) ^{\frac{1}{q}}\left\Vert \left( y_{i}^{\ast
}\right) _{i=1}^{n}\right\Vert _{r,\left\vert w\right\vert }.
\label{def1sec3}
\end{equation}%
The space of all such polynomials is denoted by $\mathcal{P}_{w,\left(
q,r\right) }^{+}\left( ^{m}E;F\right) $. Its norm is given by 
\begin{equation*}
d_{w,\left( q,r\right) }^{m+}(P)=\inf \{C>0:C\text{ \textit{satisfies }}(\ref%
{def1sec3})\}.
\end{equation*}
\end{definition}

An equivalent formulation of $(\ref{def1sec3})$ is 
\begin{equation*}
\left\Vert \left( \left\langle P\left( x_{i}\right) ,y_{i}^{\ast
}\right\rangle \right) _{i=1}^{n}\right\Vert _{p}\leq C\sup_{\phi \in B_{%
\mathcal{P}^{r}\left( ^{m}E\right) }^{+}}\left( \sum\limits_{i=1}^{n}\phi
\left( \left\vert x_{i}\right\vert \right) ^{q}\right) ^{\frac{1}{q}%
}\left\Vert \left( \left\vert y_{i}^{\ast }\right\vert \right)
_{i=1}^{n}\right\Vert _{r, w}
\end{equation*}%
for every $(x_{i})_{i=1}^{n}\subset E$ and $\left( y_{i}^{\ast }\right)
_{i=1}^{n}\subset F^{\ast }.$ It is straightforward to check that 
\begin{equation}
\mathcal{P}_{f}(^{m}E;F)\subset \mathcal{P}_{w,\left( q,r\right) }^{+}\left(
^{m}E;F\right) .  \label{4.1}
\end{equation}

\begin{proposition}
Let $P\in \mathcal{P}_{w,\left( p,r\right) }^{+}\left( ^{m}E;F\right) ,u\in 
\mathcal{L}^{+}\left( G;E\right) $ and $v\in \mathcal{L}^{+}(F;H).$ Then $%
v\circ P\circ u\in \mathcal{P}_{w,\left( q,r\right) }^{+}\left(
^{m}G;H\right) $ and we have%
\begin{equation*}
d_{w,\left( q,r\right) }^{m+}\left( v\circ P\circ u\right) \leq \left\Vert
v\right\Vert d_{w,\left( q,r\right) }^{m+}(P)\left\Vert u\right\Vert ^{m}.
\end{equation*}
\end{proposition}

\begin{proof}
Let $(x_{i})_{i=1}^{n}\subset E^{+}$ and $\left( y_{i}^{\ast }\right)
_{i=1}^{n}\subset F^{\ast +}.$ Then%
\begin{eqnarray*}
(\sum\limits_{i=1}^{n}\left\vert \left\langle v\circ P\circ u\left(
x_{i}\right) ,y_{i}^{\ast }\right\rangle \right\vert ^{p})^{\frac{1}{p}}
&=&(\sum\limits_{i=1}^{n}\left\vert \left\langle P\circ u\left(
x_{i}\right) ,v^{\ast }\circ y_{i}^{\ast }\right\rangle \right\vert ^{p})^{%
\frac{1}{p}} \\
&\leq &d_{w,\left( q,r\right) }^{m+}(P)\left\Vert \left( u\left(
x_{i}\right) \right) _{i=1}^{n}\right\Vert _{q,\left\vert w\right\vert
}^{m}\left\Vert \left( v^{\ast }\circ y_{i}^{\ast }\right)
_{i=1}^{n}\right\Vert _{r,\left\vert w\right\vert } \\
&\leq &d_{w,\left( q,r\right) }^{m+}(P)\left\Vert u\right\Vert
^{m}\left\Vert \left( x_{i}\right) _{i=1}^{n}\right\Vert _{q,\left\vert
w\right\vert }^{m}\left\Vert v^{\ast }\right\Vert \left\Vert \left(
y_{i}^{\ast }\right) _{i=1}^{n}\right\Vert _{r,\left\vert w\right\vert } \\
&\leq &\left\Vert v\right\Vert d_{w,\left( q,r\right) }^{m+}(P)\left\Vert
u\right\Vert ^{m}\left\Vert \left( x_{i}\right) _{i=1}^{n}\right\Vert
_{q,\left\vert w\right\vert }^{m}\left\Vert \left( y_{i}^{\ast }\right)
_{i=1}^{n}\right\Vert _{r,\left\vert w\right\vert }
\end{eqnarray*}%
thus $v\circ P\circ u$ is positive weakly $(q, r)$-dominated and%
\begin{equation*}
d_{w,\left( q,r\right) }^{m+}\left( v\circ P\circ u\right) \leq \left\Vert
v\right\Vert d_{w,\left( q,r\right) }^{m+}(P)\left\Vert u\right\Vert ^{m}.
\end{equation*}
\end{proof}

The pair $\left( \mathcal{P}_{w,\left( q,r\right) }^{+},d_{w,\left(
q,r\right) }^{m+}\right) $ defines a positive Banach polynomial ideal. The
proof follows directly from the previous Proposition and the inclusion $(\ref%
{4.1})$, while the remaining details are straightforward. We now turn to the
characterization of positive weakly $(q, r)$-dominated polynomials through a Pietsch
domination theorem.

\begin{theorem}[Pietsch Domination Theorem]
\label{thdo1}Let $m\in \mathbb{N}.$ Let $1\leq p,q,r\leq \infty $ with $%
\frac{1}{p}=\frac{1}{q}+\frac{1}{r}$. Let $E$ and $F$ be Banach lattices.
The following statements are equivalent.\newline
1) The polynomial $P\in \mathcal{P}\left( ^{m}E;F\right) $ is positive weakly $(q, r)$-dominated.\newline
2) There is a constant $C>0$ and Borel probability measures $\mu $ on $B_{%
\mathcal{P}^{r}\left( ^{m}E\right) }^{+}$ and $\eta $ on $B_{F^{\ast \ast
}}^{+}$ such that%
\begin{equation}
\left\vert \langle P(x),y^{\ast }\rangle \right\vert \leq C(\int_{B_{%
\mathcal{P}^{r}\left( ^{m}E\right) }^{+}}\phi \left( \left\vert x\right\vert
\right) ^{q}d\mu )^{\frac{1}{q}}(\int_{B_{F^{\ast \ast }}^{+}}\langle
\left\vert y^{\ast }\right\vert ,y^{\ast \ast }\rangle ^{r}d\eta )^{\frac{1}{%
r}}  \label{12345}
\end{equation}%
for all $(x,y^{\ast })\in E\times F^{\ast }$. Therefore, we have%
\begin{equation*}
d_{w,\left( q,r\right) }^{m+}(P)=\inf \{C>0:C\text{ \textit{satisfies }}(\ref%
{12345})\}.
\end{equation*}%
3) There is a constant $C>0$ and Borel probability measures $\mu $ on $B_{%
\mathcal{P}^{r}\left( ^{m}E\right) }^{+}$ and $\eta $ on $B_{F^{\ast \ast
}}^{+}$ such that%
\begin{equation}
\left\vert \langle P(x),y^{\ast }\rangle \right\vert \leq C(\int_{B_{%
\mathcal{P}^{r}\left( ^{m}E\right) }^{+}}\phi \left( x\right) ^{q}d\mu )^{%
\frac{1}{q}}(\int_{B_{F^{\ast \ast }}^{+}}\langle y^{\ast },y^{\ast \ast
}\rangle ^{r}d\eta )^{\frac{1}{r}}  \label{def2sec5}
\end{equation}%
for all $(x,y^{\ast })\in E^{+}\times F^{\ast +}$. Therefore, we have%
\begin{equation*}
d_{w,\left( q,r\right) }^{m+}(P)=\inf \{C>0:C\text{ \textit{satisfies }}(\ref%
{def2sec5})\}.
\end{equation*}
\end{theorem}

\begin{proof}
$1)\Longleftrightarrow 2):$ We will choose the parameters as specified in 
\cite[Theorem 4.6]{PSS12} 
\begin{equation*}
\left\{ 
\begin{array}{l}
S:\mathcal{P}\left( ^{m}E;F\right) \times \left( E\times F^{\ast }\right)
\times \mathbb{R\times R}\rightarrow \mathbb{R}^{+}: \\ 
S\left( P,\left( x,y^{\ast }\right) ,\lambda _{1},\lambda _{2}\right)
=\left\vert \lambda _{2}\right\vert |\langle P(x),y^{\ast }\rangle | \\ 
R_{1}:B_{\mathcal{P}^{r}\left( ^{m}E\right) }^{+}\times \left( E\times
F^{\ast }\right) \times \mathbb{R}\rightarrow \mathbb{R}^{+}:R_{1}(\phi
,\left( x,y^{\ast }\right) ,\lambda _{1})=\langle |x|,\phi \rangle  \\ 
R_{2}:B_{F^{\ast \ast }}^{+}\times \left( E\times F^{\ast }\right) \times 
\mathbb{R}\rightarrow \mathbb{R}^{+}:R_{2}(y^{\ast \ast },\left( x,y^{\ast
}\right) ,\lambda _{2})=\left\vert \lambda _{2}\right\vert \langle |y^{\ast
}|,y^{\ast \ast }\rangle .%
\end{array}%
\right. 
\end{equation*}%
These maps satisfy conditions $\left( 1\right) $ and $\left( 2\right) $ from 
\cite[p. 1255]{PSS12}, allowing us to conclude that $P:E\rightarrow F$ is
positive weakly $(q, r)$-dominated if, and only if, 
\begin{eqnarray*}
&&(\sum\limits_{i=1}^{n}S\left( P,\left( x_{i},y_{i}^{\ast }\right)
,\lambda _{i,1},\lambda _{i,2}\right) ^{p})^{\frac{1}{p}} \\
&\leq &C\sup_{x^{\ast }\in B_{E^{\ast
}}^{+}}(\sum\limits_{i=1}^{n}R_{1}(x^{\ast },\left( x_{i},y_{i}^{\ast
}\right) ,\lambda _{i,1})^{q})^{\frac{1}{q}}\sup_{y^{\ast \ast }\in
B_{F^{\ast \ast }}^{+}}(\sum\limits_{i=1}^{n}R_{2}(y^{\ast \ast },\left(
x_{i},y_{i}^{\ast }\right) ,\lambda _{i,2})^{r})^{\frac{1}{r}},
\end{eqnarray*}%
i.e., $P$ is $R_{1},R_{2}$-$S$-abstract $(q,r)$-summing. As outlined in \cite%
[Theorem 4.6]{PSS12}, this implies that $P$ is $R_{1},R_{2}$-$S$-abstract $%
(q,r)$-summing if, and only if, there exists a positive constant $C$ and
probability measures $\mu $ on $B_{\mathcal{P}^{r}\left( ^{m}E\right) }^{+}$
and $\eta $ on $B_{F^{\ast \ast }}^{+}$, such that 
\begin{eqnarray*}
&&S\left( P,\left( x,y^{\ast }\right) ,\lambda _{1},\lambda _{2}\right)  \\
&\leq &C(\int_{B_{E^{\ast }}^{+}}R_{1}(x^{\ast },\left( x,y^{\ast }\right)
,\lambda _{1})^{q}d\mu )^{\frac{1}{q}}(\int_{B_{F^{\ast \ast
}}^{+}}R_{2}(y^{\ast \ast },\left( x,y^{\ast }\right) ,\lambda
_{2})^{r}d\eta )^{\frac{1}{r}}.
\end{eqnarray*}%
Consequently 
\begin{equation*}
\left\vert \langle P(x),y^{\ast }\rangle \right\vert \leq C(\int_{B_{%
\mathcal{P}^{r}\left( ^{m}E\right) }^{+}}\langle \left\vert x\right\vert
,\phi \rangle ^{q}d\mu )^{\frac{1}{q}}(\int_{B_{F^{\ast \ast }}^{+}}\langle
\left\vert y^{\ast }\right\vert ,y^{\ast \ast }\rangle ^{r}d\eta )^{\frac{1}{%
r}}.
\end{equation*}%
The implications $2)\Longrightarrow 3)$ and $3)\Longrightarrow 1)$ are
straightforward to prove.
\end{proof}

\section{Tensorial representation}

Tensorial representation plays a fundamental role in the study of Banach
spaces. For the projective tensor product, it is well known that the space
of multilinear operators $\mathcal{L}\left( X_{1},\ldots,X_{m};Y\right) $
coincides with the dual of $X_{1}\widehat{\otimes }_{\pi }\cdots \widehat{%
\otimes }_{\pi }X_{m}\widehat{\otimes }_{\pi }Y^{\ast }$. Similarly,
replacing the projective norm with the injective norm $\varepsilon $ yields
a coincidence with the space of integral multilinear operators. For this
reason, tensorial representation has become a standard objective in the
study of new classes of operators. It is therefore natural, when introducing
a new class, to seek an appropriate tensor norm that produces an analogous
identification. In the present section, we define a tensor norm on the
algebraic tensor product $E_{1}\otimes \cdots \otimes E_{m}\otimes F^{\ast }$
and show that its topological dual is isometric to the space of positive
weakly $(q,r)$-dominated multilinear operators. This tensorial approach
provides a natural framework to identify such operators. In the polynomial
case, we similarly define a tensor norm on $\left( \widehat{\otimes }%
_{s,\left\vert \pi \right\vert }^{m}E\right) \otimes F$ and show that its
topological dual is isometric to the space of positive weakly $(q,r)$%
-dominated polynomials.

\subsection{Multilinear case}

Fix $m\in \mathbb{N}^{\ast }.$ Let $E_{1},\ldots,E_{m},F$ be Banach lattices.
Let $1\leq p,q,r\leq \infty $ such that $\frac{1}{p}=\frac{1}{q}+\frac{1}{r}$
and $u\in E_{1}\otimes \cdots \otimes E_{m}\otimes F.$ For simplify, we
denote by 
\begin{equation*}
\widehat{E}_{\left\vert \pi \right\vert }=E_{1}\widehat{\otimes }%
_{\left\vert \pi \right\vert }\cdots \widehat{\otimes }_{\left\vert \pi
\right\vert }E_{m}.
\end{equation*}%
Consider%
\begin{equation}
\mu _{\left( q;r\right) }^{m+}\left( u\right) =\inf \left\{ \left\Vert
\left( \lambda _{i}\right) \right\Vert _{\ell _{p^{\ast }}^{n}}\left\Vert
(x_{i}^{1}\otimes \cdots \otimes x_{i}^{m})\right\Vert _{\ell _{q,\left\vert
w\right\vert }^{n}\left( \widehat{E}_{\left\vert \pi \right\vert }\right)
}\left\Vert \left( y_{i}\right) \right\Vert _{\ell _{r,\left\vert
w\right\vert }^{n}\left( F\right) }\right\} ,  \label{3.2}
\end{equation}%
where the infimum is taken over all general representations of $u$ of the
form%
\begin{equation}
u=\sum\limits_{i=1}^{n}\lambda _{i}x_{i}^{1}\otimes \cdots \otimes
x_{i}^{m}\otimes y_{i},  \label{3.3}
\end{equation}%
with $\left( x_{i}^{j}\right) _{i=1}^{n}\subset E_{j},\left( y_{i}\right)
_{i=1}^{n}\subset F,\left( 1\leq j\leq m\right) $ and $n\in \mathbb{N}^{\ast
}.$

\begin{lemma}
\label{Lemma1}Let $u\in E_{1}\otimes \cdots \otimes E_{m}\otimes F$ of the
form $\left( \ref{3.3}\right) .$ The following properties are equivalent.%
\newline
1) $u=0.$\newline
2) $\sum\limits_{i=1}^{n}\lambda _{i}x_{1}^{\ast }\left( x_{i}^{1}\right)
\ldots.x_{m}^{\ast }\left( x_{i}^{m}\right) y^{\ast }\left( y_{i}\right) =0$
for every $x_{j}^{\ast }\subset E_{j}^{\ast +}$, $y^{\ast }\in F^{\ast
+},1\leq j\leq m$.\newline
3) $\sum\limits_{i=1}^{n}\lambda _{i}\varphi \left(
x_{i}^{1},\ldots,x_{i}^{m}\right) y^{\ast }\left( y_{i}\right) =0$ for every $%
\varphi \in \mathcal{L}^{+}\left( E_{1},\ldots ,E_{m}\right) $ and $y^{\ast
}\in F^{\ast +}$\newline
4) $\sum\limits_{i=1}^{n}\lambda _{i}\varphi \left(
x_{i}^{1},\ldots,x_{i}^{m}\right) y^{\ast }\left( y_{i}\right) =0$ for every $%
\varphi \in \mathcal{L}^{r}\left( E_{1},\ldots ,E_{m}\right) $ and $y^{\ast
}\in F^{\ast +}.$
\end{lemma}

\begin{proof}
$1)\Leftrightarrow 2):$ The first implication is straightforward by \cite[%
Proposition 1.2]{Rya}. For the second, assume $2)$ holds. Let $x_{j}^{\ast
}\in E_{j}^{\ast }$ for $1\leq j\leq m$ and $y^{\ast }\in F^{\ast }.$ Then%
\begin{eqnarray*}
&&\sum\limits_{i=1}^{n}x_{1}^{\ast }\left( x_{i}^{1}\right) \cdots
x_{m}^{\ast }\left( x_{i}^{m}\right) y^{\ast }\left( y_{i}\right)  \\
&=&\sum\limits_{i=1}^{n}\left( x_{1}^{\ast +}-x_{1}^{\ast -}\right) \left(
x_{i}^{1}\right) \cdots \left( x_{m}^{\ast +}-x_{m}^{\ast -}\right) \left(
x_{i}^{m}\right) \left( y^{\ast +}-y^{\ast -}\right) \left( y_{i}\right) .
\end{eqnarray*}%
Expanding the products yields a finite sum of terms of the form%
\begin{equation*}
\sum\limits_{i=1}^{n}x_{1}^{\ast \epsilon _{1}}\left( x_{i}^{1}\right)
\cdots x_{m}^{\ast \epsilon _{m}}\left( x_{i}^{m}\right) y^{\ast \epsilon
}\left( y_{i}\right) ,
\end{equation*}%
where $\epsilon _{1},\ldots,\epsilon _{m}\in \left\{ +,-\right\} $ and $%
\epsilon \in \left\{ +,-\right\} .$ Each of these sums corresponds to
positive functionals and, by assumption $2)$, they all vanish. Therefore,%
\begin{equation*}
\sum\limits_{i=1}^{n}x_{1}^{\ast }\left( x_{i}^{1}\right) \cdots
x_{m}^{\ast }\left( x_{i}^{m}\right) y^{\ast }\left( y_{i}\right) =0.
\end{equation*}%
Hence $u=0.$

$2)\Leftrightarrow 3):$ Assume $2)$. By the same reasoning as above, we
obtain%
\begin{equation*}
\sum\limits_{i=1}^{n}\lambda _{i}x_{1}^{\ast }\left( x_{i}^{1}\right)
\cdots x_{m}^{\ast }\left( x_{i}^{m}\right) y^{\ast }\left( y_{i}\right) =0
\end{equation*}%
for all $x_{j}^{\ast }\subset E_{j}^{\ast },y^{\ast }\in F^{\ast }.$ From 
\cite[Proposition 1.2]{Rya}, this implies $u=0,$ and consequently%
\begin{equation*}
\left\langle u,\psi \right\rangle =0\text{ for every }\left( m+1\right) 
\text{-linear form }\psi .
\end{equation*}%
In particular, this holds for $\varphi \in \mathcal{L}^{+}\left(
E_{1},\ldots ,E_{m}\right) $ and $y^{\ast }\in F^{\ast +},$ so that%
\begin{equation*}
\sum\limits_{i=1}^{n}\lambda _{i}\varphi \left(
x_{i}^{1},\ldots,x_{i}^{m}\right) y^{\ast }\left( y_{i}\right) =0\text{.}
\end{equation*}%
Conversely, if $3)$ holds, take $x_{j}^{\ast }\subset E_{j}^{\ast +}$, $%
y^{\ast }\in F^{\ast +}.$ We have $x_{1}^{\ast }\otimes \cdots \otimes
x_{m}^{\ast }\in \mathcal{L}^{+}\left( E_{1},\ldots ,E_{m}\right) .$ Hence, 
\begin{equation*}
\sum\limits_{i=1}^{n}\lambda _{i}x_{1}^{\ast }\otimes \cdots\otimes
x_{m}^{\ast }\left( x_{i}^{1},\ldots,x_{i}^{m}\right) y^{\ast }\left(
y_{i}\right) =\sum\limits_{i=1}^{n}\lambda _{i}x_{1}^{\ast }\left(
x_{i}^{1}\right) \cdots x_{m}^{\ast }\left( x_{i}^{m}\right) y^{\ast }\left(
y_{i}\right) =0.
\end{equation*}%
$3)\Longrightarrow 4):$ Let $\varphi \in \mathcal{L}^{r}\left( E_{1},\ldots
,E_{m}\right) $, there exist $T_{1},T_{2}\in \mathcal{L}^{+}\left(
E_{1},\ldots ,E_{m}\right) $ such that 
\begin{equation*}
\varphi =T_{1}-T_{2}.
\end{equation*}%
Let $y^{\ast }\in F^{\ast +},$ we have%
\begin{eqnarray*}
&&\sum\limits_{i=1}^{n}\lambda _{i}\varphi \left(
x_{i}^{1},\ldots,x_{i}^{m}\right) y^{\ast }\left( y_{i}\right)  \\
&=&\sum\limits_{i=1}^{n}\lambda _{i}T_{1}\left(
x_{i}^{1},\ldots,x_{i}^{m}\right) y^{\ast }\left( y_{i}\right)
-\sum\limits_{i=1}^{n}\lambda _{i}T_{2}\left(
x_{i}^{1},\ldots,x_{i}^{m}\right) y^{\ast }\left( y_{i}\right)  \\
&=&0.
\end{eqnarray*}%
$4)\Longrightarrow 3):$ This is immediate, since $\mathcal{L}^{+}\left(
E_{1},\ldots ,E_{m}\right) \subset \mathcal{L}^{r}\left( E_{1},\ldots
,E_{m}\right) $.
\end{proof}

The following proposition can be proved easily.

\begin{proposition}
\textit{Let }$1\leq p,q,r\leq \infty $ \textit{such that} $\frac{1}{p}=\frac{%
1}{q}+\frac{1}{r}$ \textit{and} $m\in \mathbb{N}^{\ast }.$\textit{\ Then }$%
\mu _{\left( q;r\right) }^{m+}$\textit{\ is a tensor norm on }$E_{1}\otimes
\cdots \otimes E_{m}\otimes F.$
\end{proposition}

\begin{proof}
It is clear that for any element $u\in E_{1}\otimes \cdots\otimes E_{m}\otimes
F $ of the form $\left( \ref{3.3}\right) $ and any scalar $\alpha $ we have%
\begin{equation*}
\mu _{\left( q;r\right) }^{m+}\left( u\right) \geq 0\text{ and }\mu _{\left(
q;r\right) }^{m+}\left( \alpha u\right) =\left\vert \alpha \right\vert \mu
_{\left( q;r\right) }^{m+}\left( u\right) .
\end{equation*}%
Let $\varphi \in B^{+}_{\mathcal{L}^{r}\left( E_{1},\ldots ,E_{m}\right) }$
and $y^{\ast }\in B_{F^{\ast +}}.$ Then,

$\left\vert \left\langle u,\varphi \otimes y^{\ast }\right\rangle
\right\vert $

$=\left\vert \sum\limits_{i=1}^{n}\lambda _{i}\varphi \left(
x_{i}^{1},\ldots,x_{i}^{m}\right) y^{\ast }\left( y_{i}\right) \right\vert $

$\leq $ $\sum\limits_{i=1}^{n}\left\vert \lambda _{i}\right\vert \varphi
\left( \left\vert x_{i}^{1}\right\vert ,\ldots,\left\vert x_{i}^{m}\right\vert
\right) y^{\ast }\left( \left\vert y_{i}\right\vert \right) $ by H\"{o}lder

$\leq \left\Vert \left( \lambda _{i}\right) \right\Vert _{\ell _{p^{\ast
}}^{n}}(\sum\limits_{i=1}^{n}\varphi \left( \left\vert x_{i}^{1}\right\vert
,\ldots,\left\vert x_{i}^{m}\right\vert \right) ^{q})^{\frac{1}{q}%
}(\sum\limits_{i=1}^{n}y^{\ast }\left( \left\vert y_{i}\right\vert \right)
^{r})^{\frac{1}{r}}$

$\leq \left\Vert \left( \lambda _{i}\right) \right\Vert _{\ell _{p^{\ast
}}^{n}}\sup_{\varphi \in B_{\mathcal{L}^{r}\left( E_{1},\ldots ,E_{m}\right)
}^{+}}(\sum\limits_{i=1}^{n}\varphi \left( \left\vert x_{i}^{1}\right\vert
,\ldots,\left\vert x_{i}^{m}\right\vert \right) ^{q})^{\frac{1}{q}%
}\sup_{y^{\ast }\in B_{F^{\ast +}}}(\sum\limits_{i=1}^{n}y^{\ast }\left(
\left\vert y_{i}\right\vert \right) ^{r})^{\frac{1}{r}}.$\newline
Since 
\begin{equation*}
\sup_{\varphi \in B_{\mathcal{L}^{r}\left( E_{1},\ldots ,E_{m}\right)
}^{+}}(\sum\limits_{i=1}^{n}\varphi \left( \left\vert x_{i}^{1}\right\vert
,\ldots,\left\vert x_{i}^{m}\right\vert \right) ^{q})^{\frac{1}{q}}=\left\Vert
(x_{i}^{1}\otimes \cdots \otimes x_{i}^{m})\right\Vert _{\ell _{q,\left\vert
w\right\vert }^{n}\left( \widehat{E}_{\left\vert \pi \right\vert }\right) }.
\end{equation*}%
We obtain%
\begin{equation*}
\left\vert \left\langle u,\varphi \otimes y^{\ast }\right\rangle \right\vert
\leq \left\Vert \left( \lambda _{i}\right) \right\Vert _{\ell _{p^{\ast
}}^{n}}\left\Vert (x_{i}^{1}\otimes \cdots \otimes x_{i}^{m})\right\Vert
_{\ell _{q,\left\vert w\right\vert }^{n}\left( \widehat{E}_{\left\vert \pi
\right\vert }\right) }\left\Vert \left( y_{i}\right) \right\Vert _{\ell
_{r,\left\vert w\right\vert }^{n}\left( F\right) }.
\end{equation*}%
By taking the infimum over all representations of $u,$ we obtain%
\begin{equation*}
\left\vert \left\langle u,\varphi \otimes y^{\ast }\right\rangle \right\vert
\leq \mu _{\left( q;r\right) }^{m+}\left( u\right) .
\end{equation*}%
Suppose that $\mu _{\left( q;r\right) }^{m+}\left( u\right) =0$, then 
\begin{equation*}
\left\vert \left\langle u,\varphi \otimes y^{\ast }\right\rangle \right\vert
=0,
\end{equation*}%
consequently, by Lemma \ref{Lemma1} $u=0.$ Let now $u_{1},u_{2}\in E\otimes F
$ of the form $\left( \ref{3.3}\right) $. By the definition of $\mu _{\left(
q;r\right) }^{m+},$ we can find representations%
\begin{eqnarray*}
u_{1} &=&\sum\limits_{i=1}^{s_{1}}\lambda _{1,i}x_{1,i}^{1}\otimes \cdots
\otimes x_{1,i}^{m}\otimes y_{1,i} \\
u_{2} &=&\sum\limits_{i=1}^{s_{2}}\lambda _{2,i}x_{2,i}^{1}\otimes \cdots
\otimes x_{2,i}^{m}\otimes y_{2,i}
\end{eqnarray*}%
such that%
\begin{equation*}
\left\Vert \left( \lambda _{1,i}\right) \right\Vert _{_{\ell _{p^{\ast
}}^{s_{1}}}}\left\Vert (x_{1,i}^{1}\otimes \cdots \otimes
x_{1,i}^{m})\right\Vert _{\ell _{q,\left\vert w\right\vert }^{s_{1}}\left( 
\widehat{E}_{\left\vert \pi \right\vert }\right) }\left\Vert \left(
y_{1,i}\right) \right\Vert _{\ell _{r,\left\vert w\right\vert
}^{s_{1}}\left( F\right) }\leq \mu _{\left( q;r\right) }^{m+}\left( u\right)
+\varepsilon .
\end{equation*}%
Replacing $\left( \lambda _{1,i}\right) ,\left( x_{1,i}^{1}\otimes \cdots
\otimes x_{1,i}^{m}\right) $ and $\left( y_{1,i}\right) $ by an appropriate
multiple of them,

$\lambda _{1,i}=\lambda _{1,i}\frac{\left\Vert \left( y_{1,i}\right)
\right\Vert _{\ell _{r}^{s_{1},w}\left( F\right) }^{\frac{1}{p^{\ast }}%
}\left\Vert (x_{1,i}^{1}\otimes \cdots \otimes x_{1,i}^{m})\right\Vert
_{\ell _{q,\left\vert w\right\vert }^{s_{1}}\left( \widehat{E}_{\left\vert
\pi \right\vert }\right) }^{\frac{1}{p^{\ast }}}}{\left\Vert \left( \lambda
_{1,i}\right) \right\Vert _{\ell _{p^{\ast }}^{s_{1}}}^{\frac{1}{p}}}$

$x_{1,i}^{1}\otimes \cdots \otimes x_{1,i}^{m}=x_{1,i}^{1}\otimes \cdots
\otimes x_{1,i}^{m}\frac{\left\Vert \left( \lambda _{1,i}\right) \right\Vert
_{\ell _{p^{\ast }}^{s_{1}}}^{\frac{1}{q}}\left\Vert \left( y_{1,i}\right)
\right\Vert _{\ell _{r}^{s_{1},w}\left( F\right) }^{\frac{1}{q}}}{\left\Vert
(x_{1,i}^{1}\otimes \cdots \otimes x_{1,i}^{m})\right\Vert _{\ell
_{q,\left\vert w\right\vert }^{s_{1}}\left( \widehat{E}_{\left\vert \pi
\right\vert }\right) }^{\frac{1}{q^{\ast }}}},$

$y_{1,i}=y_{1,i}\frac{\left\Vert \left( \lambda _{1,i}\right) \right\Vert
_{\ell _{p^{\ast }}^{s_{1}}}^{\frac{1}{r}}\left\Vert (x_{1,i}^{1}\otimes
\cdots \otimes x_{1,i}^{m})\right\Vert _{\ell _{q,\left\vert w\right\vert
}^{s_{1}}\left( \widehat{E}_{\left\vert \pi \right\vert }\right) }^{\frac{1}{%
r}}}{\left\Vert \left( y_{1,i}\right) \right\Vert _{\ell
_{r}^{s_{1},w}\left( F\right) }^{\frac{1}{r^{\ast }}}}.$

We can obtain%
\begin{equation*}
\begin{array}{rrr}
\left\Vert \left( \lambda _{1,i}\right) _{i=1}^{s_{1}}\right\Vert _{\ell
_{p^{\ast }}^{s_{1}}} & \leq  & \left( \mu _{\left( q;r\right) }^{m+}\left(
u_{1}\right) +\varepsilon \right) ^{\frac{1}{p^{\ast }}} \\ 
\left\Vert (x_{1,i}^{1}\otimes \cdots \otimes
x_{1,i}^{m})_{i=1}^{s_{1}}\right\Vert _{\ell _{q,\left\vert w\right\vert
}^{s_{1}}\left( \widehat{E}_{\left\vert \pi \right\vert }\right) } & \leq  & 
\left( \mu _{\left( q;r\right) }^{m+}\left( u_{1}\right) +\varepsilon
\right) ^{\frac{1}{q}}. \\ 
\left\Vert \left( y_{1,i}\right) _{i=1}^{s_{1}}\right\Vert _{\ell
_{r,\left\vert w\right\vert }^{s_{1}}\left( F\right) } & \leq  & \left( \mu
_{\left( q;r\right) }^{m+}\left( u_{1}\right) +\varepsilon \right) ^{\frac{1%
}{r}}.%
\end{array}%
\end{equation*}%
Similarly for $u_{2},$ we get%
\begin{eqnarray*}
&&\mu _{\left( q;r\right) }^{m+}\left( u_{1}+u_{2}\right)  \\
&\leq &(\left\Vert \left( \lambda _{1,i}\right) _{i=1}^{s_{1}}\right\Vert
_{\ell _{p^{\ast }}^{s_{1}}}^{p^{\ast }}+\left\Vert \left( \lambda
_{2,i}\right) _{i=1}^{s_{2}}\right\Vert _{\ell _{p^{\ast
}}^{s_{2}}}^{p^{\ast }})^{\frac{1}{p^{\ast }}}(\left\Vert
(x_{1,i}^{1}\otimes \cdots \otimes x_{1,i}^{m})_{i=1}^{s_{1}}\right\Vert
_{\ell _{q,\left\vert w\right\vert }^{s_{1}}\left( \widehat{E}_{\left\vert
\pi \right\vert }\right) }^{q} \\
&&+\left\Vert (x_{2,i}^{1}\otimes \cdots \otimes
x_{2,i}^{m})_{i=1}^{s_{2}}\right\Vert _{\ell _{q,\left\vert w\right\vert
}^{s_{2}}\left( \widehat{E}_{\left\vert \pi \right\vert }\right) }^{q})^{%
\frac{1}{q}}(\left\Vert \left( y_{1,i}\right) _{i=1}^{s_{1}}\right\Vert
_{\ell _{r,\left\vert w\right\vert }^{s_{1}}\left( F\right) }^{r}+\left\Vert
\left( y_{2,i}\right) _{i=1}^{s_{2}}\right\Vert _{\ell _{r,\left\vert
w\right\vert }^{s_{2}}\left( F\right) }^{r})^{\frac{1}{r}} \\
&\leq &\left( \mu _{\left( q;r\right) }^{m+}\left( u_{1}\right) +\mu
_{\left( q;r\right) }^{m+}\left( u_{2}\right) +2\varepsilon \right) ^{\frac{1%
}{p^{\ast }}}\left( \mu _{\left( q;r\right) }^{m+}\left( u_{1}\right) +\mu
_{\left( q;r\right) }^{m+}\left( u_{2}\right) +2\varepsilon \right) ^{\frac{1%
}{q}} \\
&&\times \left( \mu _{\left( q;r\right) }^{m+}\left( u_{1}\right) +\mu
_{\left( q;r\right) }^{m+}\left( u_{2}\right) +2\varepsilon \right) ^{\frac{1%
}{r}} \\
&\leq &\mu _{\left( q;r\right) }^{m+}\left( u_{1}\right) +\mu _{\left(
q;r\right) }^{m+}\left( u_{2}\right) +2\varepsilon .
\end{eqnarray*}%
By letting $\varepsilon $ tend to zero, we obtain the triangle inequality
for $\mu _{\left( q;r\right) }^{m+}.$
\end{proof}

\begin{proposition}
The norm $\mu _{\left( q;r\right) }^{m+}$ is reasonable, that is,%
\begin{equation}
\varepsilon \leq \mu _{\left( q;r\right) }^{m+}\leq \pi ,  \label{3.4}
\end{equation}%
\textit{\ where }$\varepsilon $ and $\pi $ denote the injective and
projective norms on $E_{1}\otimes \cdots\otimes E_{m}\otimes F,$ respectively.
\end{proposition}

\begin{proof}
Let us prove the right-hand inequality in $\left( \ref{3.4}\right) $. We have%
\begin{equation*}
\mu _{\left( q;r\right) }^{m+}\left( u\right) \leq \left\Vert \left( \lambda
_{i}\right) _{i=1}^{n}\right\Vert _{\ell _{p^{\ast }}^{n}}\left\Vert
(x_{i}^{1}\otimes \cdots \otimes x_{i}^{m})_{i=1}^{n}\right\Vert _{\ell
_{q}^{n}\left( \widehat{E}_{\left\vert \pi \right\vert }\right) }\left\Vert
\left( y_{k}\right) _{k=1}^{n_{1}}\right\Vert _{\ell _{r}^{n}\left( F\right)
}.
\end{equation*}%
For each $i$, we set%
\begin{eqnarray*}
\lambda _{i} &=&\lambda _{i}\frac{(\left\vert \lambda _{i}\right\vert
\left\Vert y_{i}\right\Vert \left\Vert x_{i}^{1}\otimes \cdots \otimes
x_{i}^{m}\right\Vert )^{\frac{1}{p^{\ast }}}}{\left\vert \lambda
_{i}\right\vert } \\
x_{1}^{i}\otimes \cdots \otimes x_{m}^{i} &=&x_{1}^{i}\otimes \cdots \otimes
x_{m}^{i}\frac{(\left\vert \lambda _{i}\right\vert \left\Vert
y_{i}\right\Vert \left\Vert x_{i}^{1}\otimes \cdots \otimes
x_{i}^{m}\right\Vert )^{\frac{1}{q}}}{\left\Vert x_{i}^{1}\otimes \cdots
\otimes x_{i}^{m}\right\Vert } \\
y_{i} &=&y_{i}\frac{(\left\vert \lambda _{i}\right\vert \left\Vert
y_{i}\right\Vert \left\Vert x_{i}^{1}\otimes \cdots \otimes
x_{i}^{m}\right\Vert )^{\frac{1}{r}}}{\left\Vert y_{i}\right\Vert }.
\end{eqnarray*}%
Substituting these expressions into the above inequality and taking the
infimum over all representations of $u$ of the form $\left( \ref{3.3}\right) 
$, we obtain%
\begin{eqnarray*}
\mu _{\left( q;r\right) }^{m+}\left( u\right)  &\leq
&\sum\limits_{i=1}^{n}\left\vert \lambda _{i}\right\vert \left\Vert
y_{i}\right\Vert \left\Vert x_{i}^{1}\otimes \cdots \otimes
x_{i}^{m}\right\Vert  \\
&\leq &\sum\limits_{i=1}^{n}\left\vert \lambda _{i}\right\vert \left\Vert
y_{i}\right\Vert \left\Vert x_{i}^{1}\right\Vert \cdots \left\Vert
x_{i}^{m}\right\Vert .
\end{eqnarray*}%
Hence,%
\begin{equation*}
\mu _{\left( q;r\right) }^{m+}\left( u\right) \leq \pi \left( u\right) .
\end{equation*}%
For the left inequality in $\left( \ref{3.4}\right) $, we have%
\begin{eqnarray*}
\varepsilon \left( u\right)  &=&\sup_{\substack{ x_{j}^{\ast }\in
B_{E_{j}^{\ast }},y^{\ast }\in B_{F^{\ast }} \\ 1\leq j\leq m}}\left\{
\left\vert \underset{i=1}{\overset{n}{\sum }}\lambda _{i}x_{1}^{\ast }\left(
x_{i}^{1}\right) \cdots x_{m}^{\ast }\left( x_{i}^{m}\right) y^{\ast }\left(
y_{i}\right) \right\vert \right\}  \\
&=&\sup_{\substack{ x_{j}^{\ast }\in B_{E_{j}^{\ast }},y^{\ast }\in
B_{F^{\ast }} \\ 1\leq j\leq m}}\left\{ \left\vert \underset{i=1}{\overset{n}%
{\sum }}\lambda _{i}x_{1}^{\ast }\otimes \cdots\otimes x_{m}^{\ast }\left(
x_{i}^{1},\ldots,x_{i}^{m}\right) y^{\ast }\left( y_{i}\right) \right\vert
\right\}. 
\end{eqnarray*}%
Since $x_{1}^{\ast }\otimes \cdots \otimes x_{m}^{\ast }\in B_{\mathcal{L}%
^{r}\left( E_{1},\ldots ,E_{m}\right) },$ and taking absolute values inside,
we get%
\begin{equation*}
\varepsilon \left( u\right) \leq \sup_{\varphi \in B_{\mathcal{L}^{r}\left(
E_{1},\ldots ,E_{m}\right) },y^{\ast }\in B_{F^{\ast }}}\left\{ \underset{i=1%
}{\overset{n}{\sum }}\left\vert \lambda _{i}\right\vert \left\vert \varphi
\left( x_{1}^{i},\ldots,x_{m}^{m}\right) \right\vert \left\vert y^{\ast }\left(
y_{i}\right) \right\vert \right\} .
\end{equation*}%
Since $\left\vert \varphi \right\vert \in B_{\mathcal{L}^{r}\left(
E_{1},\ldots ,E_{m}\right) }^{+}$ and $\left\vert y^{\ast }\right\vert \in
B_{F^{\ast }}^{+},$ we further obtain%
\begin{equation*}
\varepsilon \left( u\right) \leq \sup_{\varphi \in B_{\mathcal{L}^{r}\left(
E_{1},\ldots ,E_{m}\right) }^{+},y^{\ast }\in B_{F^{\ast +}}}\left\{ 
\underset{i=1}{\overset{n}{\sum }}\left\vert \lambda _{i}\right\vert \varphi
\left( \left\vert x_{i}^{1}\right\vert ,\ldots,\left\vert x_{i}^{m}\right\vert
\right) y^{\ast }\left( \left\vert y_{i}\right\vert \right) \right\} .
\end{equation*}%
Applying H\"{o}lder's inequality%
\begin{equation*}
\varepsilon \left( u\right) \leq \left\Vert \left( \lambda _{i}\right)
_{i=1}^{n}\right\Vert _{\ell _{p^{\ast }}^{n}}\sup_{\varphi \in B_{\mathcal{L%
}^{r}\left( E_{1},\ldots ,E_{m}\right) }^{+}}(\sum\limits_{i=1}^{n}\varphi
\left( \left\vert x_{i}^{1}\right\vert ,\ldots,\left\vert x_{i}^{m}\right\vert
\right) ^{p})^{\frac{1}{p}}\left\Vert \left( y_{i}\right)
_{i=1}^{n}\right\Vert _{\ell _{r,\left\vert w\right\vert }^{n}\left(
F\right) }.
\end{equation*}%
Finally, taking the infimum over all representations of $u$ of the form $%
\left( \ref{3.3}\right) $ yields%
\begin{equation*}
\varepsilon \left( u\right) \leq \mu _{\left( q;r\right) }^{m+}\left(
u\right) .
\end{equation*}
\end{proof}

We denote by $E_{1}\widehat{\otimes }_{\mu _{\left( q;r\right) }^{m+}}\cdots%
\widehat{\otimes }_{\mu _{\left( q;r\right) }^{m+}}E_{m}\widehat{\otimes }%
_{\mu _{\left( q;r\right) }^{m+}}F$ the completed of $E_{1}\otimes
\cdots\otimes E_{m}\otimes F$ for the norm $\mu _{\left( q;r\right) }^{m+}$.
The main result of this section is the following identification.

\begin{proposition}
\textit{Let }$1\leq p,q,r\leq \infty $ \textit{such that} $\frac{1}{p}=\frac{%
1}{q}+\frac{1}{r}.$ \textit{We have the following isometric identification}%
\begin{equation*}
\mathcal{L}_{w,\left( q,r\right) }^{m+}\left( E_{1},\ldots,E_{m};F\right)
=(E_{1}\widehat{\otimes }_{\mu _{\left( q;r\right) }^{m+}}\cdots \widehat{%
\otimes }_{\mu _{\left( q;r\right) }^{m+}}E_{m}\widehat{\otimes }_{\mu
_{\left( q;r\right) }^{m+}}F^{\ast })^{\ast }.
\end{equation*}
\end{proposition}

\begin{proof}
Let $T\in \mathcal{L}_{w,\left( q,r\right) }^{m+}\left( E,\ldots,E_{m};F\right)
.$ We define a linear functional on $E_{1}\otimes \cdots\otimes E_{m}\otimes
F^{\ast }$ by%
\begin{equation*}
\Psi _{T}\left( u\right) =\sum\limits_{i=1}^{n}\lambda _{i}\left\langle
T\left( x_{i}^{1},\ldots,x_{i}^{m}\right) ,y_{i}^{\ast }\right\rangle ,
\end{equation*}%
where $u=\sum\limits_{i=1}^{n}\lambda _{i}x_{i}^{1}\otimes \cdots \otimes
x_{i}^{m}\otimes y_{i}^{\ast }$. Then, by H\"{o}lder's inequality, we have%
\begin{eqnarray*}
\left\vert \Psi _{T}\left( u\right) \right\vert  &=&\left\vert
\sum\limits_{i=1}^{n}\lambda _{i}\left\langle T\left(
x_{i}^{1},\ldots,x_{i}^{m}\right) ,y_{i}^{\ast }\right\rangle \right\vert  \\
&\leq &\left\Vert \left( \lambda _{i}\right) \right\Vert _{\ell _{p^{\ast
}}^{n}}(\sum\limits_{i=1}^{n}\left\vert \left\langle T\left(
x_{i}^{1},\ldots,x_{i}^{m}\right) ,y_{i}^{\ast }\right\rangle \right\vert
^{p})^{\frac{1}{p}}.
\end{eqnarray*}%
Since $T$ is positive weakly $\left( q,r\right) $-domintaed, we get%
\begin{equation*}
\left\vert \Psi _{T}\left( u\right) \right\vert \leq d_{w,\left( q,r\right)
}^{m+}(T)\left\Vert \left( \lambda _{i}\right) \right\Vert _{\ell _{p^{\ast
}}^{n}}\left\Vert (x_{i}^{1}\otimes \cdots \otimes x_{i}^{m})\right\Vert
_{\ell _{q,\left\vert w\right\vert }^{n}\left( \widehat{E}_{\left\vert \pi
\right\vert }\right) }\left\Vert \left( y_{i}^{\ast }\right) \right\Vert
_{\ell _{r,\left\vert w\right\vert }^{n}\left( F^{\ast }\right) }.
\end{equation*}%
Hence, as $u$ is arbitrary, $\Psi _{T}$ is $\mu _{\left( q,r\right) }^{m+}$%
-continuous on $E_{1}\otimes \cdots\otimes E_{m}\otimes F^{\ast },$ and extends
continuously to the completed tensor product $E_{1}\widehat{\otimes }_{\mu
_{\left( q;r\right) }^{m+}}\cdots \widehat{\otimes }_{\mu _{\left(
q;r\right) }^{m+}}E_{m}\widehat{\otimes }_{\mu _{\left( q;r\right)
}^{m+}}F^{\ast }$\ with 
\begin{equation*}
\left\Vert \Psi _{T}\right\Vert \leq d_{w,\left( q,r\right) }^{m+}(T).
\end{equation*}%
Conversely, let $\Psi \in (E_{1}\widehat{\otimes }_{\mu _{\left( q;r\right)
}^{m+}}\cdots \widehat{\otimes }_{\mu _{\left( q;r\right) }^{m+}}E_{m}%
\widehat{\otimes }_{\mu _{\left( q;r\right) }^{m+}}F^{\ast })^{\ast }$. We
consider the mapping $B(\Psi )$ defined by%
\begin{equation*}
B(\Psi )\left( x^{1},\ldots,x^{m}\right) \left( y^{\ast }\right) =\Psi \left(
x^{1}\otimes \cdots \otimes x^{m}\otimes y^{\ast }\right) 
\end{equation*}%
It is clear that $B(\Psi )\in \mathcal{L}\left( E_{1},\ldots,E_{m};F\right) .$
Let $\left( x_{i}^{j}\right) _{i=1}^{n}\subset E_{j},$\ ($j=1,\ldots,m$)\ and $%
y_{1}^{\ast },\ldots,y_{n_{1}}^{\ast }\in Y^{\ast }$, we have%
\begin{eqnarray*}
&&\left\vert \sum\limits_{i=1}^{n}\left\langle B(\Psi )\left(
x_{i}^{1},\ldots,x_{i}^{m}\right) ,y_{i}^{\ast }\right\rangle \right\vert  \\
&=&\left\vert \left\langle \Psi (\sum\limits_{i=1}^{n}x_{i}^{1}\otimes
\cdots \otimes x_{i}^{m}\otimes y_{i}^{\ast }\right\rangle \right\vert  \\
&\leq &\left\Vert \Psi \right\Vert \mu _{\left( q,r\right)
}^{m+}(\sum\limits_{i=1}^{n}x_{i}^{1}\otimes \cdots \otimes
x_{i}^{m}\otimes y_{i}^{\ast }).
\end{eqnarray*}%
Therefore,%
\begin{eqnarray*}
&&(\sum\limits_{i=1}^{n}\left\vert \left\langle B(\Psi )\left(
x_{i}^{1},\ldots,x_{i}^{m}\right) ,y_{i}^{\ast }\right\rangle \right\vert
^{p})^{\frac{1}{p}} \\
&=&\sup_{\left\Vert \left( \lambda _{i}\right) \right\Vert _{\ell _{p^{\ast
}}^{n}}\leq 1}(\left\vert \sum\limits_{i=1}^{n}\lambda _{i}\left\langle
B(\Psi )\left( x_{i}^{1},\ldots,x_{i}^{m}\right) ,y_{i}^{\ast }\right\rangle
\right\vert ) \\
&=&\sup_{\left\Vert \left( \lambda _{i}\right) \right\Vert _{\ell _{p^{\ast
}}^{n}}\leq 1}(\left\vert \Psi (\sum\limits_{i=1}^{n}\lambda
_{i}x_{i}^{1}\otimes \cdots \otimes x_{i}^{m}\otimes y_{i}^{\ast
})\right\vert ) \\
&\leq &\sup_{\left\Vert \left( \lambda _{i}\right) \right\Vert _{\ell
_{p^{\ast }}^{n}}\leq 1}\left\Vert \Psi \right\Vert \left\Vert \left(
\lambda _{i}\right) \right\Vert _{\ell _{p^{\ast }}^{n_{1}}}\left\Vert
(x_{i}^{1}\otimes \cdots \otimes x_{i}^{m})\right\Vert _{\ell _{q,\left\vert
w\right\vert }^{n}\left( \widehat{E}_{\left\vert \pi \right\vert }\right)
}\left\Vert \left( y_{k}^{\ast }\right) \right\Vert _{\ell _{r,\left\vert
w\right\vert }^{n}\left( F^{\ast }\right) } \\
&\leq &\left\Vert \Psi \right\Vert \left\Vert (x_{i}^{1}\otimes \cdots
\otimes x_{i}^{m})\right\Vert _{\ell _{q,\left\vert w\right\vert }^{n}\left( 
\widehat{E}_{\left\vert \pi \right\vert }\right) }\left\Vert \left(
y_{k}^{\ast }\right) \right\Vert _{\ell _{r,\left\vert w\right\vert
}^{n}\left( F^{\ast }\right) }.
\end{eqnarray*}%
Hence, $B(\Psi )$ is positive weakly $\left( q,r\right) $-dominted and 
\begin{equation*}
d_{w,\left( q,r\right) }^{m+}\left( B(\Psi )\right) \leq \left\Vert \Psi
\right\Vert .
\end{equation*}
\end{proof}

\subsection{Polynomial case}

The polynomial case differs from the multilinear case in essential details;
it does not directly follow from the multilinear setting, and the proof
steps must be revisited to handle this situation. Let $E$ and $F$ be Banach
lattice. Let $1\leq p,q,r\leq \infty .$ Consider $u\in \left( \otimes
_{s,\left\vert \pi \right\vert }^{m}E\right) \otimes F$ of the form%
\begin{equation}
u=\sum\limits_{i=1}^{n}\lambda _{i}x_{i}\otimes \overset{\left( m\right) }{%
\cdots }\otimes x_{i}\otimes y_{i},  \label{1234}
\end{equation}%
where $\lambda _{i}\in \mathbb{R},x_{i}\in E$ and $y_{i}\in F\left( 1\leq
i\leq n\right) .$ This representation of $u$ can be considered a general
form, as any other representation can be rewritten in this way. Define%
\begin{equation*}
\lambda _{\left( q,r\right) }^{m+}\left( u\right) =\inf \left\{ \left\Vert
\left( \lambda _{i}\right) \right\Vert _{\ell _{p^{\ast }}^{n}}\sup_{\phi
\in B_{\mathcal{P}^{r}\left( ^{m}E\right) }^{+}}\left\Vert \left( \phi
\left( \left\vert x_{i}\right\vert \right) \right) \right\Vert _{\ell
_{q}^{n}}\left\Vert \left( y_{i}\right) \right\Vert _{\ell _{r,\left\vert
w\right\vert }^{n}\left( F\right) }\right\} ,
\end{equation*}%
where the infimum is taken over all general representations of $u$ of the
form $\left( \ref{1234}\right) .$

\begin{proposition}
For every $u$ of the form $\left( \ref{1234}\right) ,$ we have%
\begin{equation*}
\lambda _{\left( q,r\right) }^{m+}\left( u\right) =\mu _{\left( q,r\right)
}^{1+}\left( u\right) .
\end{equation*}
\end{proposition}

\begin{proof}
Let $u$ of the form $\left( \ref{1234}\right) .$ We have%
\begin{eqnarray*}
\mu _{\left( q,r\right) }^{1+}\left( u\right)  &=&\inf \left\{ \left\Vert
\left( \lambda _{i}\right) \right\Vert _{\ell _{p^{\ast }}^{n}}\left\Vert
(x_{i}\otimes \overset{\left( m\right) }{\cdots }\otimes x_{i})\right\Vert
_{\ell _{q,\left\vert w\right\vert }^{n}\left( \widehat{\otimes }%
_{s,\left\vert \pi \right\vert }^{m}E\right) }\left\Vert \left( y_{i}\right)
\right\Vert _{\ell _{r,\left\vert w\right\vert }^{n}\left( F\right)
}\right\}  \\
&=&\inf \left\{ \left\Vert \left( \lambda _{i}\right) \right\Vert _{\ell
_{p^{\ast }}^{n}}\sup_{\phi \in B_{\left( \widehat{\otimes }_{s,\left\vert
\pi \right\vert }^{m}E\right) ^{\ast }}^{+}}(\sum\limits_{i=1}^{n}\phi
\left( \left\vert x_{i}\otimes \overset{\left( m\right) }{\cdots }\otimes
x_{i}\right\vert \right) ^{q})^{\frac{1}{q}}\left\Vert \left( y_{i}\right)
\right\Vert _{\ell _{r,\left\vert w\right\vert }^{n}\left( F\right)
}\right\}  \\
&=&\inf \left\{ \left\Vert \left( \lambda _{i}\right) \right\Vert _{\ell
_{p^{\ast }}^{n}}\sup_{\phi \in B_{\left( \widehat{\otimes }_{s,\left\vert
\pi \right\vert }^{m}E\right) ^{\ast }}^{+}}(\sum\limits_{i=1}^{n}\phi
\left( \left\vert x_{i}\right\vert \otimes \overset{\left( m\right) }{\cdots 
}\otimes \left\vert x_{i}\right\vert \right) ^{q})^{\frac{1}{q}}\left\Vert
\left( y_{i}\right) \right\Vert _{\ell _{r,\left\vert w\right\vert
}^{n}\left( F\right) }\right\} .
\end{eqnarray*}%
Since $\mathcal{P}^{r}\left( ^{m}E\right) =\left( \widehat{\otimes }%
_{s,\left\vert \pi \right\vert }^{m}E\right) ^{\ast }$, we get%
\begin{eqnarray*}
\mu _{\left( q;r\right) }^{1+}\left( u\right)  &=&\inf \left\{ \left\Vert
\left( \lambda _{i}\right) \right\Vert _{\ell _{p^{\ast }}^{n}}\sup_{\phi
\in B_{\mathcal{P}^{r}\left( ^{m}E\right) }^{+}}\left\Vert \left( \phi
\left( \left\vert x_{i}\right\vert \right) \right) \right\Vert _{\ell
_{q}^{n}}\left\Vert \left( y_{i}\right) \right\Vert _{\ell _{r,\left\vert
w\right\vert }^{n}\left( F\right) }\right\}  \\
&=&\lambda _{\left( q,r\right) }^{m+}\left( u\right) .
\end{eqnarray*}
\end{proof}

From the above discussion on the tensor norm $\mu _{\left( q,r\right)
}^{m+}\left( u\right) $, and the previous proposition, we obtain the
following result.

\begin{corollary}
\textit{Let }$1\leq p,q,r\leq \infty $ \textit{such that} $\frac{1}{p}=\frac{%
1}{q}+\frac{1}{r}$ \textit{and} $m\in \mathbb{N}^{\ast }.$\textit{\ Then }$%
\lambda _{\left( q,r\right) }^{m+}$\textit{\ is a tensor norm on }$\left( 
\widehat{\otimes }_{s,\left\vert \pi \right\vert }^{m}E\right) \otimes F$
and we have%
\begin{equation*}
\varepsilon \leq \lambda _{\left( q,r\right) }^{m+}\leq \pi ,
\end{equation*}%
\textit{\ where }$\varepsilon $ and $\pi $ denote the injective and
projective norms on $E\otimes F,$ respectively.
\end{corollary}

We denote by $\left( \widehat{\otimes }_{s,\left\vert \pi \right\vert
}^{m}E\right) \widehat{\otimes }_{\lambda _{\left( q;r\right) }^{m+}}F$ the
completed of $\left( \otimes _{s,\left\vert \pi \right\vert }^{m}E\right)
\otimes F$ for the norm $\lambda _{\left( q;r\right) }^{m+}$.

Now, the main result of this section is the following identification.

\begin{proposition}
\textit{Let }$1\leq p,q,r\leq \infty $ \textit{such that} $\frac{1}{p}=\frac{%
1}{q}+\frac{1}{r}.$ \textit{We have the following isometric identification}%
\begin{equation*}
\mathcal{P}_{w,\left( q,r\right) }^{m+}\left( ^{m}E;F\right) =(\left( 
\widehat{\otimes }_{s,\left\vert \pi \right\vert }^{m}E\right) \widehat{%
\otimes }_{\lambda _{\left( q,r\right) }^{m+}}F^{\ast })^{\ast }.
\end{equation*}
\end{proposition}

\begin{proof}
Let $P\in \mathcal{P}_{w,\left( q,r\right) }^{m+}\left( ^{m}E;F\right) .$ We
define a linear functional on $\left( \otimes _{s,\left\vert \pi \right\vert
}^{m}E\right) \otimes F^{\ast }$ by%
\begin{equation*}
\Psi _{P}\left( u\right) =\sum\limits_{i=1}^{n}\lambda _{i}\left\langle
P\left( x_{i}\right) ,y_{i}^{\ast }\right\rangle ,
\end{equation*}%
where $u=\sum\limits_{i=1}^{n}\lambda _{i}x_{i}\otimes \cdots\otimes
x_{i}\otimes y_{i}^{\ast }$. Then, by H\"{o}lder's inequality, we have%
\begin{equation*}
\left\vert \Psi _{P}\left( u\right) \right\vert =\left\vert
\sum\limits_{i=1}^{n}\lambda _{i}\left\langle P\left( x_{i}\right)
,y_{i}^{\ast }\right\rangle \right\vert \leq \left\Vert \left( \lambda
_{i}\right) \right\Vert _{\ell _{p^{\ast
}}^{n}}(\sum\limits_{i=1}^{n}\left\langle P\left( x_{i}\right) ,y_{i}^{\ast
}\right\rangle ^{p})^{\frac{1}{p}}.
\end{equation*}%
Since $P$ is positive weakly $\left( q,r\right) $-dominated, we get%
\begin{equation*}
\left\vert \Psi _{P}\left( u\right) \right\vert \leq d_{w,\left( q,r\right)
}^{m+}(P)\left\Vert \left( \lambda _{i}\right) \right\Vert _{\ell _{p^{\ast
}}^{n}}\sup_{\phi \in B_{\mathcal{P}^{r}\left( ^{m}E\right) }^{+}}\left\Vert
\left( \phi \left( x_{i}\right) \right) \right\Vert _{\ell
_{q}^{n}}\left\Vert \left( y_{i}^{\ast }\right) \right\Vert _{\ell
_{r,\left\vert w\right\vert }^{n}\left( F^{\ast }\right) }.
\end{equation*}%
Hence, as $u$ is arbitrary, $\Psi _{P}$ is $\lambda _{\left( q,r\right)
}^{m+}$-continuous on $\left( \otimes _{s,\left\vert \pi \right\vert
}^{m}E\right) \otimes F^{\ast },$ and extends continuously to the completed
tensor product $\left( \widehat{\otimes }_{s,\left\vert \pi \right\vert
}^{m}E\right) \widehat{\otimes }_{\lambda _{\left( q,r\right) }^{m+}}F^{\ast
}$\ with 
\begin{equation*}
\left\Vert \Psi _{P}\right\Vert \leq d_{\left( q,r\right) }^{m+}(P).
\end{equation*}%
Conversely, let $\Psi \in (\left( \widehat{\otimes }_{s,\left\vert \pi
\right\vert }^{m}E\right) \widehat{\otimes }_{\lambda _{\left( q;r\right)
}^{m+}}F^{\ast })^{\ast }$. We consider the mapping $B(\Psi )$ defined by%
\begin{equation*}
B(\Psi )\left( x\right) \left( y^{\ast }\right) =\Psi \left( x\otimes 
\overset{\left( m\right) }{\cdots }\otimes x\otimes y^{\ast }\right) 
\end{equation*}%
It is clear that $B(\Psi )\in \mathcal{P}\left( ^{m}E;F\right) .$ Let $%
x_{1},\ldots,x_{n}\in E,$\ \ and $y_{1}^{\ast },\ldots,y_{n}^{\ast }\in F^{\ast }$%
, we have%
\begin{eqnarray*}
\left\vert \sum\limits_{i=1}^{n}\left\langle B(\Psi )\left( x_{i}\right)
,y_{i}^{\ast }\right\rangle \right\vert  &=&\left\vert \left\langle \Psi
(\sum\limits_{i=1}^{n}x_{i}\otimes \overset{\left( m\right) }{\cdots }%
\otimes x_{i}\otimes y_{i}^{\ast }\right\rangle \right\vert  \\
&\leq &\left\Vert \Psi \right\Vert \lambda _{\left( q,r\right)
}^{m+}(\sum\limits_{i=1}^{n}x_{i}\otimes \overset{\left( m\right) }{\cdots }%
\otimes x_{i}\otimes y_{i}^{\ast }).
\end{eqnarray*}%
Therefore,%
\begin{eqnarray*}
&&(\sum\limits_{i=1}^{n}\left\vert \left\langle B(\Psi )\left( x_{i}\right)
,y_{i}^{\ast }\right\rangle \right\vert ^{p})^{\frac{1}{p}}=\sup_{\left\Vert
\left( \lambda _{i}\right) \right\Vert _{\ell _{p^{\ast }}^{n}}\leq
1}(\left\vert \sum\limits_{i=1}^{n}\lambda _{i}\left\langle B(\Psi )\left(
x_{i}\right) ,y_{i}^{\ast }\right\rangle \right\vert ) \\
&=&\sup_{\left\Vert \left( \lambda _{i}\right) \right\Vert _{\ell _{p^{\ast
}}^{n}}\leq 1}(\left\vert \Psi (\sum\limits_{i=1}^{n}\lambda
_{i}x_{i}\otimes \overset{\left( m\right) }{\cdots }\otimes x_{i}\otimes
y_{i}^{\ast })\right\vert ) \\
&\leq &\sup_{\left\Vert \left( \lambda _{i}\right) \right\Vert _{\ell
_{p^{\ast }}^{n}}\leq 1}\left\Vert \Psi \right\Vert \left\Vert \left(
\lambda _{i}\right) \right\Vert _{\ell _{p^{\ast }}^{n}}\sup_{\phi \in B_{%
\mathcal{P}^{r}\left( ^{m}E\right) }^{+}}\left\Vert \left( \phi \left(
x_{i}\right) \right) \right\Vert _{\ell _{q}^{n}}\left\Vert \left(
y_{i}^{\ast }\right) \right\Vert _{\ell _{r,\left\vert w\right\vert
}^{n}\left( F^{\ast }\right) } \\
&\leq &\left\Vert \Psi \right\Vert \sup_{\phi \in B_{\mathcal{P}^{r}\left(
^{m}E\right) }^{+}}\left\Vert \left( \phi \left( x_{i}\right) \right)
\right\Vert _{\ell _{q}^{n}}\left\Vert \left( y_{i}^{\ast }\right)
\right\Vert _{\ell _{r,\left\vert w\right\vert }^{n}\left( F^{\ast }\right)
}.
\end{eqnarray*}%
Hence, $B(\Psi )$ is positive weakly $\left( q,r\right) $-dominted and 
\begin{equation*}
d_{w,\left( q,r\right) }^{m+}\left( B(\Psi )\right) \leq \left\Vert \Psi
\right\Vert .
\end{equation*}
\end{proof}

\end{document}